\documentclass[11pt,reqno,a4paper]{amsart}

\usepackage[T1]{fontenc}
\usepackage{lmodern}
\usepackage{microtype}
\usepackage{amsmath,amssymb,amsthm,mathtools}
\usepackage{enumitem}
\usepackage[colorlinks=true,linkcolor=blue,citecolor=blue,urlcolor=blue]{hyperref}
\hypersetup{
 pdftitle={Parameter-uniform Robin uniqueness on large dilations},
 pdfauthor={Sophie Sun and Xuanrui Zhang}
}

\numberwithin{equation}{section}

\newtheorem{theorem}{Theorem}[section]
\newtheorem{proposition}[theorem]{Proposition}
\newtheorem{lemma}[theorem]{Lemma}
\newtheorem{corollary}[theorem]{Corollary}
\theoremstyle{definition}

\theoremstyle{remark}
\newtheorem{remark}[theorem]{Remark}

\newcommand{\R}{\mathbb R}
\newcommand{\Hh}{\mathbb H}
\newcommand{\cN}{\mathcal N}
\newcommand{\cU}{\mathcal U}
\newcommand{\dist}{\operatorname{dist}}

\title[Parameter-uniform Robin uniqueness]
{Parameter-uniform Robin uniqueness on large dilations}

\author{Sophie Sun}
\email{sophiesun1@yeah.net}

\author{Xuanrui Zhang}
\email{xruizhang1@126.com}


\begin{document}

\begin{abstract}
Berestycki and Graham proved large-dilation uniqueness for bounded positive
solutions of
\[
 -\Delta u=f(u)\quad\hbox{in }\kappa\Omega,
 \qquad u+\alpha\partial_\nu u=0\quad\hbox{on }\partial(\kappa\Omega),
\]
when $\alpha$ is fixed, and remarked that the dilation threshold should not
depend on $\alpha$.  We show that it does not, including at the Dirichlet and
Neumann endpoints, for possibly unbounded uniformly $C^{2,\gamma}$ domains.
The half-space linearizations have a common positive spectral gap over the
compactified boundary parameter.  The Dirichlet end requires a separate
compactness argument because the Robin coefficient diverges there.  After
rescaling by its inverse, the equation has a harmonic half-space limit.  A
Liouville lemma rules out a nonzero limiting trace, and the resulting endpoint
compactness, together with the half-space gap and localization, yields the
uniform uniqueness statement.
When $\partial\Omega\neq\varnothing$, we further obtain convergence of the
spectral bottom to its half-space value with error $O(\kappa^{-1/2})$.
For bounded $C^{4,\gamma}$ domains the boundary layer also has a first
mean-curvature correction, with remainder
$O(\kappa^{-1-\gamma}+\kappa^{-2})$ on each fixed boundary strip.
\end{abstract}

\maketitle

\section{Introduction}

We consider bounded positive solutions of
\begin{equation}\label{eq:alpha-problem}
 \begin{cases}
  -\Delta u=f(u)&\text{in }\kappa\Omega,\\
  u+\alpha\partial_\nu u=0&\text{on }\partial(\kappa\Omega),
 \end{cases}
\end{equation}
where $d\ge1$, $\nu$ is the outward unit normal, $\alpha\ge0$, and $\kappa>0$.
For some $\gamma\in(0,1]$ we assume
\begin{equation}\label{eq:f-assumptions}
 \begin{gathered}
 f\in C^{1,\gamma}([0,\infty)),\qquad f(0)=f(1)=0,\\
 f>0\ \text{on }(0,1),\qquad f'(0)>0,\qquad f'(1)<0,
 \qquad f<0\ \text{on }(1,\infty).
 \end{gathered}
\end{equation}
We work with the positive-reaction class of
\cite{BerestyckiGrahamStableCompact,BerestyckiGrahamStrongKPP}.  The domain
$\Omega$ is open and connected, uniformly $C^{2,\gamma}$ in the sense of
\cite[Definition~A.1]{BerestyckiGrahamStrongKPP}, and may be unbounded.  Solutions
are classical, with $u\in C^2(\kappa\Omega)\cap C^1(\overline{\kappa\Omega})\cap L^\infty(\kappa\Omega)$;
uniform smoothness includes a uniform interior-ball condition.

For each fixed Robin parameter, Proposition~8.1 of
\cite{BerestyckiGrahamStableCompact} gives uniqueness once the dilation is
large, with a threshold that may depend on $\alpha$.  Immediately afterward,
Berestycki and Graham ask whether this threshold can be chosen independently of
$\alpha$.  Theorem~\ref{thm:main} gives a parameter-uniform answer in their
general-domain positive-reaction setting, including the two limiting boundary
conditions.

\begin{theorem}\label{thm:main}
Assume \eqref{eq:f-assumptions}, and let $\Omega\subset\R^d$ be an open,
connected, uniformly $C^{2,\gamma}$ domain.  There exists $\kappa_0=\kappa_0(f,\Omega)>0$ such that, for
every $\kappa\ge\kappa_0$ and every finite $\alpha\ge0$,
\eqref{eq:alpha-problem} has exactly one bounded positive solution.  The same
threshold remains valid at the Neumann endpoint $\alpha=+\infty$, interpreted as
$\partial_\nu u=0$.
\end{theorem}

Rossi's Corollary~2.8 implies that at the Neumann endpoint $u\equiv1$ is
the unique bounded positive solution \cite{RossiStability}.  To treat the boundary
condition uniformly,
we use the compact parameter
\begin{equation}\label{eq:compact-param}
 \varrho=\frac{\alpha}{1+\alpha}\in[0,1),\qquad
 \cN_\varrho u:=\varrho\partial_\nu u+(1-\varrho)u,
\end{equation}
and adjoin $\varrho=1$.  Thus $\varrho=0$ is Dirichlet,
$0<\varrho<1$ is Robin, and $\varrho=1$ is Neumann.  This compactification is
standard in the Robin theory \cite{BerestyckiGrahamStrongKPP}.

There is also a uniform version for compact families of reactions.  Let
$\mathfrak F$ consist of positive reactions satisfying
\eqref{eq:f-assumptions}.  We call $\mathfrak F$ \emph{compactly positive} when
\begin{equation}\label{eq:compact-reaction-family}
 \{f|_{[0,1]}:f\in\mathfrak F\}\text{ is compact in }C^1([0,1]),
 \qquad
 \sup_{f\in\mathfrak F}[f']_{C^{0,\gamma}([0,1])}<\infty.
\end{equation}
In compactness arguments we identify $\mathfrak F$ with this compact set of
restrictions, endowed with the $C^1([0,1])$ topology.

\begin{theorem}
\label{thm:reaction-family}
Let $\mathfrak F$ be compactly positive and let $\Omega$ be as above.  There exists
$\kappa_0=\kappa_0(\mathfrak F,\Omega)>0$ such that the conclusion of
Theorem~\ref{thm:main} holds simultaneously for every $f\in\mathfrak F$.
\end{theorem}

The half-space profiles are translates of the Dirichlet layer and converge to
$1$ as $\varrho\uparrow1$.  The estimate needed below is the uniform positivity
of the linearized spectral bottom:
\begin{equation}\label{eq:intro-gap}
 \inf_{\varrho\in[0,1]}
 \lambda\bigl(-\partial_y^2-f'(\phi_\varrho),\R_+,\varrho\bigr)>0.
\end{equation}
Only the Dirichlet endpoint is singular: there
$\beta_\varrho\to\infty$ in $\partial_\nu u+\beta_\varrho u=0$.  At the scale
$\beta_\varrho^{-1}$ the limit problem is harmonic in a half-space.  The
Liouville lemma in Section~\ref{sec:compactness} rules out a nonzero limiting
trace, and Robin--Lieb localization then transfers the half-space gap to the
dilated domain.

When $\partial\Omega\ne\varnothing$, write
\[
 d_\kappa(x)=\dist(x,\partial(\kappa\Omega)),\qquad
 \Phi_{\kappa,\varrho}(x)=\phi_\varrho(d_\kappa(x)),
\]
and set
$\mu_\varrho=\lambda(-\partial_y^2-f'(\phi_\varrho),\R_+,\varrho)$.  Then, for all sufficiently large $\kappa$,
\begin{equation}\label{eq:intro-spectral-convergence}
 \sup_{\varrho\in[0,1]}
 \left|\lambda\bigl(-\Delta-f'(\Phi_{\kappa,\varrho}),
                \kappa\Omega,\varrho\bigr)-\mu_\varrho\right|
 \le C\kappa^{-1/2}.
\end{equation}
In the proof we take $R\asymp\sqrt\kappa$, balancing the Fermi-chart error
$O(R/\kappa)$ against the $O(R^{-1})$ quasimode truncation error.

When $\Omega$ is bounded and $C^{4,\gamma}$, the first correction to the
boundary layer can also be identified explicitly.  Write $D_\kappa=\kappa\Omega$ and let
$u_{\kappa,\varrho}$ be the unique state.  In a fixed boundary strip, $\pi_\kappa(x)$ denotes the nearest
boundary point.  With $\mathcal H$ the sum of the outward principal curvatures,
define $\chi_\varrho$ as the unique $H^1(\R_+)$ solution of
\begin{equation}\label{eq:intro-curvature-response}
 \begin{cases}
 \bigl(-\partial_y^2-f'(\phi_\varrho)\bigr)\chi_\varrho=-\phi_\varrho',\\
 -\varrho\chi_\varrho'(0)+(1-\varrho)\chi_\varrho(0)=0.
 \end{cases}
\end{equation}
At $\varrho=1$ one has $\chi_1=0$.

\begin{theorem}
\label{thm:curvature-correction}
Assume \eqref{eq:f-assumptions} and that $\Omega$ is bounded, connected, and
$C^{4,\gamma}$.  For every fixed $R>0$ there are $C_R>0$ and
$\kappa_R>0$ such that, for every $\kappa\geq\kappa_R$,
\begin{align}
 &\sup_{\varrho\in[0,1]}
 \sup_{\substack{x\in D_\kappa\\d_\kappa(x)\le R}}
 \left|u_{\kappa,\varrho}(x)-\phi_\varrho(d_\kappa(x))
 -\frac{\mathcal H(\pi_\kappa(x)/\kappa)}{\kappa}
       \chi_\varrho(d_\kappa(x))\right| \notag\\
 &\hspace{35mm}\le C_R(\kappa^{-1-\gamma}+\kappa^{-2}).
 \label{eq:intro-curvature-expansion}
\end{align}
Consequently,
\begin{equation}\label{eq:intro-scaled-curvature-limit}
 \sup_{\varrho\in[0,1]}
 \sup_{\substack{x\in D_\kappa\\d_\kappa(x)\le R}}
 \left|\kappa\bigl(u_{\kappa,\varrho}(x)-\phi_\varrho(d_\kappa(x))\bigr)
 -\mathcal H(\pi_\kappa(x)/\kappa)\chi_\varrho(d_\kappa(x))\right|
 \longrightarrow0.
\end{equation}
\end{theorem}

Robin singular perturbations and boundary-layer asymptotics have been studied
in several settings
\cite{HowesRobinNeumann,BerestyckiWeiRobin,KimRobinAsymptotic,LyuLinPoissonBoltzmann,LeeMoonYangNonlocal}.
Here the boundary scale is created by geometric dilation, and the profile and
remainder are uniform over the compactified Dirichlet--Robin--Neumann family.
Recent work on Robin uniqueness and half-line asymptotics includes
\cite{ChenGrossiLiRobin,ChenGrossiLiRobinAsymptotics,ChuWongHalfLine}.  The two
recent papers of Chen, Grossi, and Li concern power nonlinearities on fixed
bounded domains as the Robin parameter varies.  There is also a large-parameter
literature in which the parameter enters both the equation and the boundary
condition.  In particular, Fonseka, Shivaji, Son, and Spetzer
\cite{FonsekaShivajiSonSpetzer} study
$-\Delta u=\lambda f(u)$ on bounded domains with a prescribed boundary
coefficient $\mu(\lambda)$, including finite carrying-capacity reactions.

For bounded $\Omega$, the change of variables $x=\kappa y$ rewrites
\eqref{eq:alpha-problem}, for $0<\alpha<\infty$, as
$-\Delta U=\kappa^2f(U)$ in $\Omega$ with
$\partial_\nu U+(\kappa/\alpha)U=0$ on $\partial\Omega$.  Thus the finite-Robin
bounded-domain subproblem lies in the same large-parameter scaling as the work of
Hai, Shivaji, and Wang \cite{HaiShivajiWangRobin}.  For their reaction class they
obtain a large-$\lambda$ uniqueness threshold independent of every finite Robin
coefficient $b>0$; their hypotheses are stated for globally nonnegative,
sublinear reactions satisfying an eventual concavity-type condition.  Since every
positive state under \eqref{eq:f-assumptions} takes values in $(0,1]$, changing
the reaction above $1$ does not change the states considered here.  Consequently,
the bounded finite-Robin subcase can overlap, after such a modification, with
large-parameter results of this type, and we do not regard that subcase alone as
the novelty of Theorem~\ref{thm:main}.  The point here is the uniform statement
on uniformly smooth, possibly unbounded domains, together with a single
compactification that includes the Dirichlet and Neumann endpoints.  The common
spectral gap, the two-sided uniform spectral convergence, and the curvature
expansion are further conclusions beyond the uniqueness statement.

\section{Spectral setup for the compactified parameter}\label{sec:prelim}

For a domain $D$ and $\varrho\in[0,1]$, let $\cN_\varrho$ be the operator from
\eqref{eq:compact-param}.  When $\varrho\in(0,1]$, set
\begin{equation}\label{eq:beta}
 \beta_\varrho:=\frac{1-\varrho}{\varrho}\in[0,\infty).
\end{equation}
For $\varrho>0$, the condition $\cN_\varrho u=0$ is equivalent to
$\partial_\nu u+\beta_\varrho u=0$.  We interpret
$\beta_0=+\infty$ as Dirichlet data.

Let $V\in L^\infty(D)$ be real.  For $\varrho>0$ define the closed
quadratic form
\begin{equation}\label{eq:robin-form-global}
 \mathfrak q_{V,\varrho}[\psi]
 :=\int_D\bigl(|\nabla\psi|^2+V\psi^2\bigr)
   +\beta_\varrho\int_{\partial D}\psi^2,
 \qquad \psi\in H^1(D),
\end{equation}
and at $\varrho=0$ use the same expression without the boundary term on
$H_0^1(D)$.  Uniform smoothness gives the global trace bound needed for the
Robin form.  We write
\begin{equation}\label{eq:form-bottom-definition}
 \lambda(-\Delta+V,D,\varrho)
 :=\inf_{\psi\ne0}
 \frac{\mathfrak q_{V,\varrho}[\psi]}{\|\psi\|_{L^2(D)}^2},
\end{equation}
with the appropriate form domain understood.  For uniformly H\"older
potentials this agrees with the generalized principal eigenvalue of
\cite[Proposition~2.1(v)]{BerestyckiGrahamStrongKPP}; see also
\cite{BerestyckiRossi} for the Dirichlet theory on unbounded domains.  We use
three standard form facts repeatedly:
\begin{enumerate}[label=(\roman*)]
 \item monotonicity under Dirichlet truncation;
 \item the perturbation estimate
 \begin{equation}\label{eq:potential-lipschitz}
  \bigl|\lambda(-\Delta+V_1,D,\varrho)
  -\lambda(-\Delta+V_2,D,\varrho)\bigr|
  \leq\|V_1-V_2\|_{L^\infty(D)};
 \end{equation}
 \item positivity of $\lambda(-\Delta+V,D,\varrho)$ implies the bounded
 maximum principle for $\Delta-V$ with homogeneous $\cN_\varrho$ boundary
 data.  The classical formulation in
 \cite[Proposition 3.1]{BerestyckiGrahamStrongKPP} assumes H\"older
 coefficients and classical solutions; the bounded-measurable, weak-solution
 version used below is recorded in Lemma~\ref{lem:bounded-weak-maximum}.
\end{enumerate}
The underlying Robin form theory on general domains can be found in
\cite{Daners}.

\begin{lemma}
\label{lem:bounded-weak-maximum}
Let $D$ be a uniformly smooth domain, let $V\in L^\infty(D)$ be real, and
suppose
\[
 \lambda(-\Delta+V,D,\varrho)>0.
\]
If $w\in H^1_{\rm loc}(\overline D)\cap L^\infty(D)$ is a weak
supersolution of $(-\Delta+V)w\geq0$ with homogeneous $\cN_\varrho$
boundary condition, then $w\geq0$.  Here ``weak'' means the variational
inequality associated with the Robin form when $\varrho>0$; at
$\varrho=0$ we require the negative part of the trace to vanish and use
zero-trace test functions.
\end{lemma}

\begin{proof}
The cutoff proof behind
\cite[Proposition 3.1]{BerestyckiGrahamStrongKPP} adapts to the
quadratic-form setting; we include the form calculation because it is used
later with merely bounded potentials.  Put
$\zeta_\epsilon(x)=\operatorname{sech}(\epsilon|x|)$, so that
$|\nabla\zeta_\epsilon|\leq\epsilon\zeta_\epsilon$ almost everywhere, and
let $\eta_R$ be a standard cutoff with $0\leq\eta_R\leq1$,
$\eta_R=1$ on $B_R$, $\eta_R=0$ outside $B_{2R}$, and
$|\nabla\eta_R|\leq C/R$.  Set
$a_R:=\zeta_\epsilon\eta_R$.  Using $a_R^2w_-$ as a nonnegative test
function in the supersolution inequality and restricting to the set
$\{w<0\}$ gives, for $\varrho>0$,
\[
 \int_D \nabla w_-\cdot\nabla(a_R^2w_-)
 +\int_D Va_R^2w_-^2
 +\beta_\varrho\int_{\partial D}a_R^2w_-^2\leq0,
\]
and the same formula without the boundary term holds at $\varrho=0$.
(The Dirichlet trace assumption makes the test function admissible there.)
The elementary identity
\[
 |\nabla(a_Rw_-)|^2
 =\nabla w_-\cdot\nabla(a_R^2w_-)+|\nabla a_R|^2w_-^2
\]
therefore yields
\begin{equation}\label{eq:weighted-negative-part}
 \mathfrak q_{V,\varrho}[a_Rw_-]
 \leq\int_D|\nabla a_R|^2w_-^2.
\end{equation}
Since $V\in L^\infty(D)$, the boundary term in the Robin form is
nonnegative, and $w$ is bounded, \eqref{eq:weighted-negative-part} gives a
uniform $H^1(D)$ bound for $a_Rw_-$.  Indeed,
$\zeta_\epsilon\in L^2(\R^d)$, while the contribution containing
$\nabla\eta_R$ is supported in $B_{2R}\setminus B_R$ and tends to zero
exponentially as $R\to\infty$.

Thus, after extraction, $a_Rw_-$ converges weakly in $H^1$ and strongly in
$L^2$ to $\zeta_\epsilon w_-$.  Moreover,
\[
 \int_D|\nabla a_R|^2w_-^2
 \longrightarrow \int_D|\nabla\zeta_\epsilon|^2w_-^2;
\]
the cross term and the $\nabla\eta_R$ term vanish by the same exponential-tail
estimate.  Closedness of the Robin form (and of $H_0^1(D)$ at
$\varrho=0$), together with weak lower semicontinuity, now gives
\[
 \mathfrak q_{V,\varrho}[\zeta_\epsilon w_-]
 \leq\int_D |\nabla\zeta_\epsilon|^2w_-^2
 \leq\epsilon^2\|\zeta_\epsilon w_-\|_2^2.
\]
The homogeneous boundary condition is already encoded in the form.  The
Rayleigh inequality gives the reverse lower bound
\[
 \mathfrak q_{V,\varrho}[\zeta_\epsilon w_-]
 \geq\lambda(-\Delta+V,D,\varrho)
       \|\zeta_\epsilon w_-\|_2^2.
\]
Choosing $\epsilon^2<\lambda(-\Delta+V,D,\varrho)$ forces $w_-=0$.
\end{proof}

In localized problems, the physical boundary inside $B_R(z)$ keeps the
$\varrho$ condition and the spherical boundary is Dirichlet; points in
$\partial D\cap\partial B_R(z)$ belong to the artificial boundary.  We write
\[
 \lambda(-\Delta+V,D\mid B_R(z),\varrho).
\]
This denotes the bottom of the mixed form on $D\cap B_R(z)$; for a
disconnected intersection it is the infimum over components, and it is
$+\infty$ when the intersection is empty.  For potentials depending only on
the normal variable, the half-space bottom is one-dimensional.

\begin{lemma}\label{lem:product}
Let $W\in L^\infty(\R_+)$ be real and let $\varrho\in[0,1]$.  Then
\begin{equation}\label{eq:product}
 \lambda(-\Delta+W(y),\Hh^d,\varrho)
 =\lambda(-\partial_y^2+W(y),\R_+,\varrho),
\end{equation}
where $\Hh^d=\R^{d-1}\times\R_+$.
\end{lemma}

\begin{proof}
The case $d=1$ is immediate.  Suppose $d\ge2$.  Applying the half-line
Rayleigh inequality to $\psi(x',\cdot)$ for almost every $x'$ and then
integrating in the tangential variables gives one inequality, since the
remaining tangential energy is nonnegative.  Conversely, multiply a half-line
test function $h(y)$ by an $L^2$-normalized tangential cutoff $\eta_R(x')$
whose Dirichlet energy tends to zero.  Letting $R\to\infty$ gives the opposite
inequality.  The same argument applies at the Dirichlet endpoint.
\end{proof}

Theorem~2.3 of \cite{BerestyckiGrahamStrongKPP} gives the Robin localization
estimate for the Laplacian.  The same Fubini argument includes a bounded
potential; for the potential-bearing Dirichlet version see
\cite[Theorem~3.5]{BerestyckiGrahamStableCompact}.  We record the resulting
form explicitly.

\begin{theorem}\label{thm:lieb}
Let $D$ be a uniformly $C^{2,\gamma}$ domain, let $V\in L^\infty(D)$ be real,
and let $\varrho\in[0,1]$.  For every $R>0$,
\begin{equation}\label{eq:lieb}
 \lambda(-\Delta+V,D,\varrho)
 \geq \inf_{z\in\R^d}
 \lambda(-\Delta+V,D\mid B_R(z),\varrho)
 -\lambda_1(B_1)R^{-2},
\end{equation}
where empty intersections are assigned eigenvalue $+\infty$.
\end{theorem}

\begin{proof}
For $\varrho>0$ write $\beta=\beta_\varrho$.  Fix $\varepsilon>0$ and take
an $L^2(D)$-normalized form function $a$ whose quotient is below
$\lambda(-\Delta+V,D,\varrho)+\varepsilon/2$.  We also take
$b\in C_c^\infty(B_R)$ with $\|b\|_2=1$ and
$\int|\nabla b|^2<\lambda_1(B_R)+\varepsilon/2$.  At the Dirichlet endpoint
$a$ has zero trace and the boundary terms below are omitted.

Set $h_z(x)=a(x)b(x-z)$ and let $A(z)=\|h_z\|_2^2$ and $T(z)$ be its mixed
form energy on $D\mid B_R(z)$.  Fubini gives $\int_{\R^d}A(z)\,dz=1$.
The potential and Robin terms localize exactly.  For the gradient term,
\[
 |\nabla(ab_z)|^2
 =|\nabla a|^2b_z^2+a^2|\nabla b_z|^2
   +\tfrac12\nabla(a^2)\cdot\nabla(b_z^2),
\]
and the last term has zero integral in $z$, since
$\nabla(b_z^2)=-\nabla_z(b(x-z)^2)$ and $b$ is compactly supported.  Thus
\[
 \int_{\R^d}T(z)\,dz
 <\lambda(-\Delta+V,D,\varrho)+\lambda_1(B_R)+\varepsilon.
\]
Since $\int_{\R^d}A(z)\,dz=1$, some $z$ with $A(z)>0$ satisfies
$T(z)/A(z)$ below the right-hand side.  The mixed Rayleigh principle, followed
by $\varepsilon\downarrow0$ and
$\lambda_1(B_R)=\lambda_1(B_1)R^{-2}$, gives \eqref{eq:lieb}.
\end{proof}

\section{Half-space layers and their spectral gap}\label{sec:profiles}

At the boundary blow-up scale the profile is one-dimensional.  The fixed-parameter
ODE profile and half-space classification are known
\cite[Lemma~6.1 and Theorem~1.1(A)]{BerestyckiGrahamHalfSpace}; here we need
uniform control as the parameter runs from Dirichlet to Neumann.  Let
\begin{equation}\label{eq:P}
 P(s):=\int_s^1 f(r)\,dr,
 \qquad 0\leq s\leq1.
\end{equation}
Then $P(s)>0$ for $s<1$.  For $\varrho\in[0,1)$, consider
\begin{equation}\label{eq:halfline-profile}
 \begin{cases}
  -\phi_\varrho''=f(\phi_\varrho),&y>0,\\
  -\varrho\phi_\varrho'(0)+(1-\varrho)\phi_\varrho(0)=0,\\
  \phi_\varrho(+\infty)=1.
 \end{cases}
\end{equation}
The outward normal at $y=0$ is $-\partial_y$.  We set
$\phi_1\equiv1$.

\begin{proposition}\label{prop:translation}
For every $\varrho\in[0,1)$, problem
\eqref{eq:halfline-profile} has a unique solution satisfying
$0\leq\phi_\varrho<1$ and $\phi_\varrho'>0$.  If $\phi_0$ is the Dirichlet
profile, there is a unique $c_\varrho\in[0,\infty)$ such that
\begin{equation}\label{eq:translation}
 \phi_\varrho(y)=\phi_0(y+c_\varrho).
\end{equation}
The map $\varrho\mapsto\phi_\varrho$ is continuous from
$[0,1]$ into $C_b([0,\infty))$, and
\begin{equation}\label{eq:uniform-tail}
 \phi_0(y)\leq\phi_\varrho(y)\leq1,
 \qquad
 \sup_{\varrho\in[0,1]}\bigl(1-\phi_\varrho(y)\bigr)
 =1-\phi_0(y)\longrightarrow0.
\end{equation}
\end{proposition}

\begin{proof}
For the Dirichlet profile the quadrature is explicit.  Define
\[
 Y(s):=\int_0^s\frac{dr}{\sqrt{2P(r)}}.
\]
Since $f'(1)<0$, one has
$P(s)\sim |f'(1)|(1-s)^2/2$ as $s\nearrow1$, so $Y$ is strictly
increasing from $[0,1)$ onto $[0,\infty)$.  Its inverse $\phi_0$ satisfies
\begin{equation}\label{eq:energy}
 \phi_0'(y)=\sqrt{2P(\phi_0(y))},
 \qquad -\phi_0''=f(\phi_0),
\end{equation}
as well as $\phi_0(0)=0$, $\phi_0'>0$, and $\phi_0(y)\to1$.

For $\varrho\in(0,1)$ put $\alpha=\varrho/(1-\varrho)$ and consider
\[
 H(s):=\frac{s}{\sqrt{2P(s)}}.
\]
It is strictly increasing on $[0,1)$, since
\[
 H'(s)=\frac1{\sqrt{2P(s)}}
 +\frac{s f(s)}{(2P(s))^{3/2}}>0.
\]
Since $H(0)=0$ and $H(s)\to\infty$ as $s\nearrow1$, there is a unique
$s_\varrho\in(0,1)$ such that
\begin{equation}\label{eq:s-alpha}
 H(s_\varrho)=\alpha.
\end{equation}
It depends continuously and increasingly on $\alpha$, and
$s_\varrho\to1$ as $\varrho\nearrow1$.  Set
\[
 c_\varrho:=Y(s_\varrho),
 \qquad \phi_\varrho(y):=\phi_0(y+c_\varrho).
\]
The identities \eqref{eq:energy} and \eqref{eq:s-alpha} give
$\alpha\phi_\varrho'(0)=s_\varrho=\phi_\varrho(0)$, exactly the boundary
condition in \eqref{eq:halfline-profile}.  At $\varrho=0$ we take $c_0=0$.

For uniqueness, take another profile in the stated class.  Multiplying the ODE
by its positive derivative and using the limit at infinity gives the same first
integral,
\[
 \phi_\varrho'(y)=\sqrt{2P(\phi_\varrho(y))}.
\]
For $0<\varrho<1$, the boundary trace must solve
\eqref{eq:s-alpha}; for $\varrho=0$, the trace is $s_0=0$.  In either case,
separation of variables gives
$Y(\phi_\varrho(y))=y+Y(s_\varrho)$.  Uniqueness and \eqref{eq:translation} follow.  Continuity for $\varrho<1$ follows from continuity
of $c_\varrho$ and uniform continuity of $\phi_0$.  As
$\varrho\nearrow1$, monotonicity gives
\[
 \sup_{y\geq0}|1-\phi_\varrho(y)|=1-s_\varrho\longrightarrow0.
\]
This proves continuity at the Neumann endpoint and
\eqref{eq:uniform-tail}.  The corresponding half-space uniqueness theorem is
\cite[Theorem 1.1(A)]{BerestyckiGrahamHalfSpace}; all limits to which we apply
it take values in $[0,1]$, so only the restriction of $f$ to that interval is
relevant.
\end{proof}

\begin{lemma}\label{lem:halfline-classification}
Let $\varrho\in[0,1)$, and let
$u\in C^2(\R_+)\cap C^1([0,\infty))$ be a bounded, nonzero, nonnegative
solution of
\[
 -u''=f(u),\qquad
 -\varrho u'(0)+(1-\varrho)u(0)=0,
 \qquad 0\leq u\leq1.
\]
Then $u=\phi_\varrho$.
\end{lemma}

\begin{proof}
If $\varrho=0$, then $u(0)=0$ and ODE uniqueness gives $u'(0)>0$;
if $0<\varrho<1$, then
$u'(0)=\beta_\varrho u(0)>0$, since otherwise the initial data would force
$u\equiv0$.  The bound $u\leq1$ and uniqueness for the ODE prevent $u$
from reaching $1$ at a finite point: at a first such point a positive derivative
would cross above $1$, while a zero derivative would force the constant solution
$u\equiv1$.  So $0<u<1$ on $(0,\infty)$.  On this interval
$u''=-f(u)<0$, so $u'$ is strictly decreasing.  It cannot vanish: after a first
zero it would become negative and, being decreasing while $u$ remains in
$(0,1)$, would force $u$ eventually to become negative.  So $u'$ stays positive on
$\R_+$.  The bounded increasing function $u$ has a limit $\ell\in(0,1]$, and $u'\downarrow0$.  If $f(\ell)>0$, then
$u''=-f(u)$ is bounded above by a negative constant for all sufficiently large
y, contradicting $u'\geq0$.  We must have $f(\ell)=0$, and positivity of $f$ on
$(0,1)$ gives $\ell=1$.  The uniqueness part of Proposition~\ref{prop:translation} now gives
$u=\phi_\varrho$.
\end{proof}

\begin{lemma}
\label{lem:halfline-form-continuity}
Let $V\in L^\infty(\R_+)$ be real.  For $\beta\in[0,\infty)$ set
\begin{equation}\label{eq:halfline-form}
 \nu_\beta(V):=
 \inf_{h\in H^1(\R_+)\setminus\{0\}}
 \frac{\displaystyle
  \int_0^\infty\bigl(|h'|^2+Vh^2\bigr)\,dy
  +\beta |h(0)|^2}
 {\displaystyle\int_0^\infty h^2\,dy},
\end{equation}
and let $\nu_\infty(V)$ be the same infimum over
$H_0^1(\R_+)\setminus\{0\}$, without the boundary term.  Then
$\beta\mapsto\nu_\beta(V)$ extends continuously to the compactified
interval $[0,\infty]$.  More quantitatively, there is $C_V>0$ such that
\begin{equation}\label{eq:quantitative-robin-dirichlet}
 0\leq \nu_\infty(V)-\nu_\beta(V)\leq C_V\beta^{-1/2}
 \qquad(\beta\geq1).
\end{equation}
\end{lemma}

\begin{proof}
Start with $\beta$ in a bounded interval $[0,B]$.  A single $H^1$ test function gives a uniform upper bound for
$\nu_\beta(V)$ on $[0,B]$, while
$\nu_\beta(V)\geq-\|V_-\|_\infty$.  For every
$\varepsilon\in(0,1]$, each $L^2$-normalized $\varepsilon$-minimizer
$h$ for a parameter $\beta\in[0,B]$ satisfies
\[
 \|h'\|_2^2\leq C_{B,V}.
\]
The one-dimensional trace estimate
\[
 |h(0)|^2\leq 2\|h\|_2\|h'\|_2
\]
gives $|h(0)|^2\leq C_{B,V}$.  If $h_2$ is such an
$\varepsilon$-minimizer for $\beta_2$, then using the same function in the
$\beta_1$ quotient gives
\[
 \nu_{\beta_1}(V)
 \leq \nu_{\beta_2}(V)+\varepsilon
       +(\beta_1-\beta_2)|h_2(0)|^2.
\]
Interchanging $\beta_1$ and $\beta_2$ and then letting
$\varepsilon\downarrow0$ gives, on $[0,B]$,
\[
 |\nu_{\beta_1}(V)-\nu_{\beta_2}(V)|
 \leq C_{B,V}|\beta_1-\beta_2|.
\]
Thus $\nu_\beta(V)$ is continuous for finite $\beta$.  At the Dirichlet
endpoint the Robin penalty diverges, so we use a different comparison.  Since $H_0^1(\R_+)\subset H^1(\R_+)$ and
its functions have zero trace,
\begin{equation}\label{eq:robin-below-dirichlet}
 \nu_\beta(V)\leq\nu_\infty(V).
\end{equation}
Fix $\beta\geq1$ and choose $h\in H^1(\R_+)$ with
$\|h\|_2=1$ and
\[
 \int_0^\infty\bigl(|h'|^2+Vh^2\bigr)\,dy
 +\beta|h(0)|^2
 \leq \nu_\beta(V)+\beta^{-1}.
\]
By \eqref{eq:robin-below-dirichlet}, the right-hand side is uniformly
bounded for $\beta\geq1$.  Both $\|h'\|_2$ and $\beta|h(0)|^2$ are uniformly bounded.  Writing
$a=h(0)$, we have $|a|\leq C\beta^{-1/2}$.  Define
\[
 g(y):=h(y)-ae^{-y}.
\]
Then $g\in H_0^1(\R_+)$ and
$\|g-h\|_{H^1}\leq C\beta^{-1/2}$.  The uniform bound on $\|h'\|_2$,
together with $V\in L^\infty$, gives
\[
 \bigl|\|g\|_2^2-1\bigr|\leq C\beta^{-1/2},
 \qquad
 \left|\int_0^\infty\bigl(|g'|^2+Vg^2\bigr)
 -\int_0^\infty\bigl(|h'|^2+Vh^2\bigr)\right|
 \leq C\beta^{-1/2}.
\]
For all sufficiently large $\beta$, the Dirichlet Rayleigh principle and
the preceding estimates yield
\[
 \nu_\infty(V)\leq \nu_\beta(V)+C_V\beta^{-1/2}.
\]
After enlarging $C_V$, the same estimate covers the remaining bounded range
of $\beta\geq1$.  Together with \eqref{eq:robin-below-dirichlet}, this is
\eqref{eq:quantitative-robin-dirichlet}; in particular the form bottom is
continuous at the Dirichlet endpoint.
\end{proof}

For $\varrho\in[0,1]$, let
\begin{equation}\label{eq:mu-rho}
 \mu_\varrho:=
 \lambda\bigl(-\partial_y^2-f'(\phi_\varrho),\R_+,\varrho\bigr).
\end{equation}

\begin{proposition}\label{prop:uniform-gap}
The map $\varrho\mapsto\mu_\varrho$ is continuous on $[0,1]$ and
\begin{equation}\label{eq:uniform-gap}
 \mu_*:=\min_{\varrho\in[0,1]}\mu_\varrho>0.
\end{equation}
Lemma~\ref{lem:product} transfers the same lower bound to the
linearized operators on $\Hh^d$.
\end{proposition}

\begin{proof}
Take $\varrho<1$ and write
\[
 L_\varrho:=-\partial_y^2+V_\varrho,
 \qquad V_\varrho(y):=-f'(\phi_\varrho(y)),
 \qquad c_\infty:=-f'(1)>0.
\]
On the half-line, the form bottom in \eqref{eq:mu-rho} is the bottom of the
self-adjoint spectrum of $L_\varrho$ with the indicated boundary condition.
Since $V_\varrho-c_\infty\to0$ at infinity, multiplication by
$V_\varrho-c_\infty$ is a relatively compact perturbation of the constant
coefficient half-line operator with the same boundary condition.  To see this
directly, split $V_\varrho-c_\infty$ into its restriction to $[0,R]$ and its
tail.  The compactly supported part composed with the resolvent is compact by
local $H^2$ regularity and Rellich compactness, whereas the operator norm of
the tail tends to zero as $R\to\infty$.  Weyl's theorem therefore gives
\[
 \sigma_{\rm ess}(L_\varrho)=[c_\infty,\infty).
\]
Test functions translated to infinity show $\mu_\varrho\leq c_\infty$.
If $\mu_\varrho<c_\infty$, one-dimensional Sturm--Liouville theory gives a
simple isolated eigenvalue with an $L^2$ eigenfunction $\psi>0$ on
$(0,\infty)$; the case $\mu_\varrho=c_\infty$ is immediate.

Differentiating the profile equation yields
\begin{equation}\label{eq:zero-mode}
 \bigl(-\partial_y^2-f'(\phi_\varrho)\bigr)\phi_\varrho'=0.
\end{equation}
The spectral separation from $c_\infty$ gives exponential decay of
$\psi$ and $\psi'$.  The energy identity and
$P(s)\sim c_\infty(1-s)^2/2$ give exponential decay of
$\phi_\varrho'$ and $\phi_\varrho''$.  The Wronskian vanishes at infinity,
so integration gives
\begin{equation}\label{eq:wronskian}
 \mu_\varrho\int_0^\infty\psi\phi_\varrho'
 =\psi'(0)\phi_\varrho'(0)-\psi(0)\phi_\varrho''(0).
\end{equation}
The integral on the left is strictly positive.  At $\varrho=0$, one has
$\psi(0)=0$ and $\psi'(0)>0$, so the right-hand side equals
$\psi'(0)\phi_0'(0)>0$.  If
$0<\varrho<1$, then
\[
 \psi'(0)=\beta_\varrho\psi(0),
 \qquad \phi_\varrho''(0)=-f(s_\varrho),
\]
and $\psi(0)>0$ by uniqueness for the ODE initial-value problem.  The right-hand side of \eqref{eq:wronskian} is
\[
 \psi(0)\bigl(\beta_\varrho\phi_\varrho'(0)+f(s_\varrho)\bigr)>0.
\]
$\mu_\varrho>0$ for every $\varrho<1$.  At the Neumann endpoint
$\phi_1\equiv1$, and
\[
 \mu_1=\lambda(-\partial_y^2+|f'(1)|,\R_+,1)=|f'(1)|>0.
\]

To make the lower bound uniform in $\varrho$, use
Proposition~\ref{prop:translation} and the uniform continuity of $f'$:
$V_\varrho$ depends continuously on $\varrho$ in $L^\infty(\R_+)$.  For
$\varrho>0$, the boundary condition is the Robin form condition with
$\beta_\varrho=(1-\varrho)/\varrho$, while $\varrho=0$ is the
$\beta=\infty$ Dirichlet endpoint.  Lemma
\ref{lem:halfline-form-continuity}, together with
\eqref{eq:potential-lipschitz}, gives
\[
 \mu_{\varrho_j}\longrightarrow\mu_\varrho
 \qquad\text{whenever }\varrho_j\to\varrho\in[0,1].
\]
Thus $\varrho\mapsto\mu_\varrho$ is positive and continuous on $[0,1]$,
and \eqref{eq:uniform-gap} follows by compactness.
\end{proof}

\begin{proposition}
\label{prop:family-halfline}
Let $\mathfrak F$ be compactly positive.  Then
\begin{equation}\label{eq:family-profile-continuity}
 (f,\varrho)\longmapsto\phi^f_\varrho
 \quad\text{is continuous from }\mathfrak F\times[0,1]
 \text{ into }C_b([0,\infty)),
\end{equation}
and $(f,\varrho)\mapsto\mu^f_\varrho$ is continuous.  On the compact
parameter set this gives
\begin{equation}\label{eq:family-uniform-gap}
 \widehat\mu:=\min_{f\in\mathfrak F}\min_{\varrho\in[0,1]}
 \mu^f_\varrho>0.
\end{equation}
The profile tails can be chosen uniformly in $(f,\varrho)$.
\end{proposition}

\begin{proof}
The compact family has uniform endpoint data:
\begin{equation}\label{eq:family-basic-data}
 \begin{gathered}
 a_0:=\min_f f'(0)>0,\qquad
 a_1:=\min_f(-f'(1))>0,\\
 A_-:=\min_f\int_0^1 f(s)\,ds>0,\\
 M_1:=\sup_f\|f\|_{C^1([0,1])}<\infty,\qquad
 M_\gamma:=\sup_f[f']_{C^{0,\gamma}([0,1])}<\infty.
 \end{gathered}
\end{equation}
For each compact $I\Subset(0,1)$, positivity and compactness also give
$\inf_f\inf_I f>0$.  The Dirichlet quadrature then has a uniform finite
entrance time into a fixed neighborhood of $1$, while $a_1$ and
$M_\gamma$ give a common exponential tail there.  Since every Robin profile is a
translate of its Dirichlet profile, for some $C,c>0$,
\begin{equation}\label{eq:family-common-exponential-tail}
 \sup_{f\in\mathfrak F}\sup_{\varrho\in[0,1]}
 (1-\phi^f_\varrho(y))\le Ce^{-cy}.
\end{equation}
If $f_n\to f_*$ in $C^1([0,1])$ and $\varrho_n\to\varrho_*$, the defining
boundary-trace equation gives trace convergence when $0<\varrho_*<1$; the
uniform positivity of $P_f$ away from $1$ handles $\varrho_*=1$, and the common
bound on $P_f(0)$ handles $\varrho_*=0$.  The quadrature and the preceding common
tail then give $\phi^{f_n}_{\varrho_n}\to\phi^{f_*}_{\varrho_*}$ uniformly.  Hence the
potentials converge in $L^\infty$, and Lemma~\ref{lem:halfline-form-continuity}
plus \eqref{eq:potential-lipschitz} gives continuity of the spectral bottoms.
Pointwise positivity follows from Proposition~\ref{prop:uniform-gap}; compactness
then yields \eqref{eq:family-uniform-gap}.
\end{proof}

\section{Endpoint compactness and coalescence}
\label{sec:compactness}

Oblique estimates are uniform away from the Dirichlet endpoint, but fail as
$\varrho\downarrow0$ because $\beta_\varrho\to\infty$.  At the reciprocal
scale the limit is harmonic; the following Liouville lemma controls its trace.

\begin{lemma}\label{lem:harmonic-robin}
Let $v\in C^2(\Hh^d)\cap C^1(\overline{\Hh^d})$ be bounded and harmonic.
If
\begin{equation}\label{eq:unit-robin}
 \partial_\nu v+v=0\quad\text{on }\partial\Hh^d,
\end{equation}
then $v\equiv0$.
\end{lemma}

\begin{proof}
If $d=1$, then $v$ is an affine function on $\R_+$; boundedness makes it
constant, and the boundary condition makes that constant zero.  Assume
$d\geq2$ and write $\Hh^d=\{(x',y):y>0\}$.  Then
\eqref{eq:unit-robin} is $\partial_yv=v$ at $y=0$.
Set $g=v(\cdot,0)\in L^\infty(\R^{d-1})$.  The Poisson extension of $g$
is bounded and has trace $g$.  Subtracting this extension from $v$ gives a bounded harmonic function, continuous up to
$\{y=0\}$ and vanishing there.  Its odd reflection is harmonic in the
distributional sense on $\R^d$; boundedness and Liouville's theorem therefore make
that difference vanish.  In semigroup notation,
\[
 v(\cdot,y)=e^{-y|D|}g,
 \qquad |D|=(-\Delta_{x'})^{1/2}.
\]
The boundary condition says, in distributions,
\begin{equation}\label{eq:fractional-trace}
 (|D|+1)g=0.
\end{equation}
Apply the Poisson semigroup to \eqref{eq:fractional-trace}.  The function
$q(t):=e^{-t|D|}g$ satisfies $q'(t)=q(t)$ in tempered distributions, and
so $q(t)=e^tg$.  On the other hand, the Poisson semigroup is a
contraction on $L^\infty$, so
\[
 e^t\|g\|_\infty=\|q(t)\|_\infty\leq\|g\|_\infty
 \qquad(t>0).
\]
This forces $g=0$, and the Poisson representation gives $v=0$.
\end{proof}

To pass from vanishing traces to convergence on closed boundary patches we use
the following uniform boundary decay estimate.

\begin{lemma}\label{lem:boundary-decay}
Let $G_j$ be domains whose boundaries in $B_{2r}(p_j)$ are given by
uniformly $C^{1,1}$ graphs, and let
$h_j\in L^\infty(G_j\cap B_{2r}(p_j))$.  Suppose that
\[
 \begin{gathered}
 w_j\in H^1(G_j\cap B_{2r}(p_j))
 \cap C(\overline{G_j\cap B_{2r}(p_j)}),\\
 -\Delta w_j=h_j\quad\hbox{in }G_j\cap B_{2r}(p_j),\\
 \|w_j\|_\infty\leq M,\qquad \|h_j\|_\infty\leq M_1.
 \end{gathered}
\]
There are $\eta\in(0,1)$ and $C>0$, depending only on the dimension and
the uniform graph bounds, such that, for $0<\delta<r/8$,
\begin{align}
 \sup_{\substack{x\in G_j\cap B_{r/2}(p_j)\\
                  \dist(x,\partial G_j)\leq\delta}}
 |w_j(x)|
 &\leq \sup_{\partial G_j\cap B_r(p_j)}|w_j|
 \notag\\
 &\quad+C(M+M_1r^2)\left(\frac{\delta}{r}\right)^\eta .
 \label{eq:uniform-boundary-decay}
\end{align}
The constants are uniform over the family $G_j$.
\end{lemma}

\begin{proof}
Set
\[
 a_j=\sup_{\partial G_j\cap B_r(p_j)}|w_j|,
 \qquad z_j^\pm=(\pm w_j-a_j)_+ .
\]
By Kato's inequality, $-\Delta z_j^\pm\leq M_1$ weakly,
$z_j^\pm=0$ on the physical boundary, and $0\leq z_j^\pm\leq M$.  For
$x\in G_j\cap B_{r/2}(p_j)$ with
$d_x=\dist(x,\partial G_j)\leq r/8$, choose a nearest boundary point $q$.
Then $|q-p_j|\leq5r/8$, and hence $B_{r/4}(q)\subset B_r(p_j)$.  Translate
$q$ to the origin and rescale by $r/4$.  On the rescaled unit-size patch the
boundary is represented by a $C^{1,1}$ graph with norms bounded uniformly in
$j$, the functions $z_j^\pm$ vanish on the physical part of the boundary,
and
\[
 -\Delta z_j^\pm\leq M_1,
 \qquad 0\leq z_j^\pm\leq M.
\]
After rescaling, the right-hand side is bounded by $CM_1r^2$.  The standard
one-sided boundary H\"older (equivalently, boundary growth) estimate for
nonnegative weak subsolutions with zero trace, applied after flattening the
boundary; see \cite[Section~8.10]{GilbargTrudinger}, therefore gives
\[
 z_j^\pm(x)\leq
 C(M+M_1r^2)\left(\frac{d_x}{r}\right)^\eta .
\]
Both the exponent and the constant depend only on the dimension and on the
uniform $C^{1,1}$ character of the graph family.  Applying the estimate to
both signs proves \eqref{eq:uniform-boundary-decay}.
\end{proof}

We combine the smooth-convergence framework of
\cite[Appendix A]{BerestyckiGrahamStrongKPP} with
Lemma~\ref{lem:boundary-decay} near the boundary.

\begin{lemma}\label{lem:joint-compactness}
Let $D_n$ be a uniformly $C^{2,\gamma}$ sequence converging locally in
$C^{2,\gamma'}$, for some $\gamma'<\gamma$, to a domain $D_\infty$.  Let
$\varrho_n\to\varrho_\infty\in[0,1]$, and suppose
$u_n\in C^2(D_n)\cap C^1(\overline{D_n})$ satisfies
\begin{equation}\label{eq:compactness-pde}
 0\leq u_n\leq1,\qquad -\Delta u_n=f(u_n)\ \text{in }D_n,
 \qquad \cN_{\varrho_n}u_n=0\ \text{on }\partial D_n.
\end{equation}
After passage to a subsequence, $u_n$ converges locally uniformly on the
closure of each component that persists under this local smooth convergence
to a function $u_\infty$ satisfying
\begin{equation}\label{eq:limit-pde}
 -\Delta u_\infty=f(u_\infty)\ \text{in }D_\infty,
 \qquad \cN_{\varrho_\infty}u_\infty=0\ \text{on }\partial D_\infty.
\end{equation}
\end{lemma}

\begin{proof}
Interior elliptic estimates already give subsequential
$C^2_{\rm loc}$ convergence.  At a limit parameter
$\varrho_\infty>0$, the boundary causes no loss of compactness: after dividing
by $\varrho_n$, the Robin coefficients $\beta_{\varrho_n}$ stay bounded.  Standard
local oblique estimates in flattened charts then give convergence up to the
boundary, including the Neumann case, and the limit satisfies
\eqref{eq:limit-pde}.

Suppose now that $\varrho_\infty=0$.  Interior compactness is unchanged, so it
remains to exclude a nonzero boundary trace as the Robin coefficient diverges.  Terms with $\varrho_n=0$ are already covered by
ordinary Dirichlet compactness, so after discarding that case and taking a
subsequence we may assume $\varrho_n>0$ for every $n$.
Then
\[
 \beta_n:=\beta_{\varrho_n}\longrightarrow\infty,
 \qquad \partial_\nu u_n+\beta_nu_n=0.
\]
The boundary traces must vanish locally.  Otherwise some fixed chart
contains points $p_n\in\partial D_n$ and there is $\varepsilon>0$ such that
$u_n(p_n)\geq\varepsilon$.  Translate $p_n$ to
the origin and choose an orthogonal map $Q_n$ such that
$Q_ne_d=-\nu_{D_n}(p_n)$.  Define the rescaled domain and function by
\[
 E_n:=\{z:p_n+\beta_n^{-1}Q_nz\in D_n\},
 \qquad
 v_n(z):=u_n\bigl(p_n+\beta_n^{-1}Q_nz\bigr).
\]
If the normalized boundary graph for $D_n$ in these coordinates is
$y=h_n(x')$, then
the graph for $E_n$ is
\[
 y=\widehat h_n(z'):=\beta_nh_n(z'/\beta_n).
\]
Because the graphs are normalized by $h_n(0)=0$ and
$\nabla h_n(0)=0$, the uniform $C^{2,\gamma}$ bounds give, on every fixed
ball $B_L\subset\R^{d-1}$,
\[
 \|\widehat h_n\|_{C^1(B_L)}
 +\|D^2\widehat h_n\|_{L^\infty(B_L)}
 \leq C_L\beta_n^{-1}\longrightarrow0.
\]
The chart radii are enlarged by the factor $\beta_n$, while the rescaled
graphs flatten, and $E_n\to\Hh^d$ locally in $C^{1,1}$, including the
boundary.  If
\[
 \widehat\nu_n(z):=Q_n^T\nu_{D_n}
 \bigl(p_n+\beta_n^{-1}Q_nz\bigr),
\]
then the chain rule gives
\begin{equation}\label{eq:rescaled}
 -\Delta v_n=\beta_n^{-2}f(v_n),\qquad
 \partial_{\widehat\nu_n}v_n+v_n=0,\qquad 0\leq v_n\leq1.
\end{equation}
In these variables the Robin coefficient is $1$ and the right-hand side in
\eqref{eq:rescaled} is $O(\beta_n^{-2})$, so the obliqueness constants are
uniform.  The one-dimensional case is elementary.  Suppose $d\geq2$, fix
$p>d$, and fix $L'<L$.  For large $n$, the portion
$\partial E_n\cap B_L$ is a graph with arbitrarily small Lipschitz constant.
On this physical boundary we have $b_0=1$ and
$b=\widehat\nu_n=\nu_{E_n}$, so
$b\cdot\nu_{E_n}=1$; moreover the $C^{0,1}$ norms of
$\widehat\nu_n$ on fixed balls are uniformly bounded (in fact their
Lipschitz seminorms are $O(\beta_n^{-1})$).  Thus the boundary coefficients
satisfy uniformly the hypotheses of
\cite[Theorem~2.3]{DongLiOblique}; since $b$ is Lipschitz, it belongs to
$C^{0,\alpha}$ for every $\alpha<1$, and we choose
$\alpha>1-1/p$.

There is no circular regularity issue in applying the a priori estimate:
for each fixed $n$, the rescaled boundary is still $C^{2,\gamma}$ and the
classical solution $v_n$ has the usual local $W^{2,p}$ regularity up to the
oblique boundary.  We use Dong--Li only to make the constants uniform in
$n$.  More precisely, apply their localized estimate (5.3) to a bounded
patch between $B_{L'}$ and $B_L$, taking as the distinguished boundary
portion the part inherited from $\partial E_n$.  Its remaining artificial
boundary stays a fixed positive distance from $E_n\cap B_{L'}$.  Since
$0\leq v_n\leq1$,
\[
 \|v_n\|_{L^p}\leq C_L,
 \qquad
 \|\Delta v_n\|_{L^p}\leq C_L\beta_n^{-2},
 \qquad
 \|\partial_{\widehat\nu_n}v_n+v_n\|_{W^{1-1/p,p}}=0
\]
on that patch and its physical boundary.  The constants in (5.3) depend
only on the fixed geometric and coefficient bounds just listed.  Hence
\[
 \|v_n\|_{W^{2,p}(E_n\cap B_{L'})}\leq C_{L,L',p}.
\]
Sobolev embedding yields uniform $C^{1,\theta}$ bounds on compact half-balls.
A diagonal subsequence yields a limit
$v\in C^1_{\rm loc}(\overline{\Hh^d})\cap C^2(\Hh^d)$, and
\eqref{eq:rescaled} passes to
\[
 \Delta v=0,\qquad \partial_\nu v+v=0,\qquad
 0\leq v\leq1,\qquad v(0)\geq\varepsilon.
\]
Lemma~\ref{lem:harmonic-robin} rules out such a limit.  Consequently,
\begin{equation}\label{eq:trace-vanish}
 \sup_{K\cap\partial D_n}u_n\longrightarrow0
\end{equation}
for every fixed compact set $K$ in a persistent chart.

Interior estimates give a solution $u_\infty$ in $D_\infty$.
Apply Lemma~\ref{lem:boundary-decay} in finitely many smaller charts covering a
fixed compact subset of $\partial D_\infty$.  Write $p_{n,j}$ for the
corresponding centers on $\partial D_n$, and take a common chart radius
$r>0$.  With $M=1$ and $M_1=\|f\|_{L^\infty(0,1)}$, equations
\eqref{eq:trace-vanish} and
\eqref{eq:uniform-boundary-decay} imply, for each $j$,
\begin{equation}\label{eq:boundary-layer-small}
 \limsup_{n\to\infty}
 \sup_{\substack{x\in D_n\cap B_{r/2}(p_{n,j})\\
                 \dist(x,\partial D_n)\leq\delta}}u_n(x)
 \leq C(1+M_1r^2)\left(\frac{\delta}{r}\right)^\eta.
\end{equation}
Equation~\eqref{eq:boundary-layer-small} shows that $u_\infty$ extends by zero
to the limiting boundary.  Combining the same estimate on the boundary layer
with interior compactness away from it, and then letting $\delta\downarrow0$,
gives locally uniform convergence on the closure.
\end{proof}

For $\kappa>0$, put $D_\kappa=\kappa\Omega$.  When
$\partial\Omega\neq\varnothing$, define
\begin{equation}\label{eq:Phi}
 d_\kappa(x):=\dist(x,\partial D_\kappa),
 \qquad
 \Phi_{\kappa,\varrho}(x):=\phi_\varrho(d_\kappa(x)).
\end{equation}
When $\Omega=\R^d$, we set $\Phi_{\kappa,\varrho}\equiv1$ and do not use
$d_\kappa$.
Let $\cU_{\kappa,\varrho}$ be the set of bounded positive solutions of
\begin{equation}\label{eq:rho-problem}
 -\Delta u=f(u)\ \text{in }D_\kappa,
 \qquad \cN_\varrho u=0\ \text{on }\partial D_\kappa.
\end{equation}
The bounded maximum principle \cite{BerestyckiGrahamStrongKPP} gives
$0<u\leq1$ for every such solution.

From this point through Section~\ref{sec:proof-main}, $\Omega$ is the open,
connected, uniformly $C^{2,\gamma}$ domain fixed in
Theorem~\ref{thm:main}.  Unless stated otherwise, constants attached to the
dilated domains may depend on $f$ and on the uniform $C^{2,\gamma}$ geometry
of $\Omega$, but not on $\kappa$ or $\varrho$.

\begin{lemma}\label{lem:interior-saturation}
Given $s\in(0,1)$, there is $R=R(f,s)>0$ with the following property.  If
$D\subset\R^d$ is a domain, $u$ is a bounded positive solution of
$-\Delta u=f(u)$ in $D$ with $0<u\leq1$, and $B_R(x)\subset D$, then
\begin{equation}\label{eq:interior-lower}
 u(x)\geq s.
\end{equation}
The radius $R$ is independent of $D$ and of any boundary condition imposed on
$\partial D$.
\end{lemma}

\begin{proof}
Set
\[
 m_s:=\inf_{0<r\leq s}\frac{f(r)}r>0
\]
and choose $R$ so large that $\lambda_1(B_{R/2})<m_s$.  Let $v$ be the
positive first Dirichlet eigenfunction of $B_{R/2}$, normalized so that
$\max v=v(0)=s$.  Then every multiple $\theta v$, $0\leq\theta\leq1$, is a
subsolution in $B_{R/2}$, because
\[
 -\Delta(\theta v)=\lambda_1(B_{R/2})\theta v
 \leq m_s\theta v\leq f(\theta v).
\]
Translate the smaller ball so that it is centered at $x$.  Since
$\overline{B_{R/2}(x)}\subset B_R(x)\subset D$, positivity and continuity of
$u$ place a small multiple of $v(\,\cdot-x)$ below $u$ on the closed smaller
ball; slide that multiple upward toward $1$.
A first contact must be interior, because $v=0<u$ on
$\partial B_{R/2}(x)$.  At such a contact, $u=\theta v\in(0,s]$ and
\[
 -\Delta(u-\theta v)
 =f(\theta v)-\theta\lambda_1(B_{R/2})v>0,
\]
contradicting the second-derivative test at an interior minimum.  Thus
$u(x)\geq v(0)=s$.  This is the local comparison used in
\cite[Lemma 2.7]{BerestyckiGrahamStableCompact}; no boundary condition on
$\partial D$ enters.
\end{proof}

\begin{proposition}\label{prop:coalescence}
One has
\begin{equation}\label{eq:coalescence}
 \lim_{\kappa\to\infty}
 \sup_{\varrho\in[0,1]}
 \sup_{u\in\cU_{\kappa,\varrho}}
 \|u-\Phi_{\kappa,\varrho}\|_{L^\infty(D_\kappa)}=0.
\end{equation}
The supremum over an empty solution set is understood as zero.
\end{proposition}

\begin{proof}
If $\Omega=\R^d$, Lemma~\ref{lem:interior-saturation} applied at every point
and for every $s<1$ gives $u\equiv1=\Phi_{\kappa,\varrho}$, so there is no
boundary layer to analyze.  Assume from now on that
$\partial\Omega\neq\varnothing$ and argue by contradiction.  If
\eqref{eq:coalescence} were false, then for some
$\varepsilon>0$ there
are
\[
 \kappa_n\to\infty,\qquad \varrho_n\in[0,1],\qquad
 u_n\in\cU_{\kappa_n,\varrho_n},\qquad x_n\in D_{\kappa_n}
\]
such that
\begin{equation}\label{eq:bad-point}
 |u_n(x_n)-\Phi_{\kappa_n,\varrho_n}(x_n)|\geq\varepsilon.
\end{equation}
Choose a subsequence for which $\varrho_n\to\varrho_*$.  Along a further
subsequence, either $d_{\kappa_n}(x_n)\to\infty$ or
$\{d_{\kappa_n}(x_n)\}_n$ is bounded.

If $d_{\kappa_n}(x_n)\to\infty$, fix any $s\in(0,1)$.  For all large $n$,
$d_{\kappa_n}(x_n)\geq R(f,s)$, so Lemma~\ref{lem:interior-saturation} gives
$u_n(x_n)\geq s$.  Since $s<1$ is arbitrary and $u_n\leq1$,
\[
 u_n(x_n)\longrightarrow1.
\]
On the other hand, \eqref{eq:uniform-tail} gives
\[
 \Phi_{\kappa_n,\varrho_n}(x_n)
 =\phi_{\varrho_n}(d_{\kappa_n}(x_n))\longrightarrow1,
\]
contradicting \eqref{eq:bad-point}.

Thus the bad points cannot escape into the bulk; their distance from the
boundary stays bounded.  Choose a nearest boundary point $p_n$.
Because $\partial D_{\kappa_n}$ is $C^1$, first variation of squared distance
along the boundary gives
\[
 x_n-p_n=d_{\kappa_n}(x_n)e_n,
\]
where $e_n$ is the inward unit normal at $p_n$.  Translate by $p_n$ and rotate
$e_n$ to $e_d$.  The uniform boundary-chart scale of $D_{\kappa_n}$ tends to
infinity while the rescaled $C^2$ graph norms tend to zero, so the domains
converge locally to $\Hh^d$.  Lemma~\ref{lem:joint-compactness} gives a locally
uniform limit $u_*$ solving the $\varrho_*$ problem on $\Hh^d$.

The limit is nonzero.  Choose $R_0>R(f,1/2)$.  By the uniform interior-ball
condition, each $p_n$ has a tangent interior ball of radius
$L_n=\kappa_n r_{\rm in}$.  For large $n$, $L_n>R_0$, and with
$y_n=p_n+R_0e_n$ elementary geometry gives
\[
 B_{R_0}(y_n)\subset D_{\kappa_n},\qquad
 d_{\kappa_n}(y_n)=R_0.
\]
Hence Lemma~\ref{lem:interior-saturation} gives $u_n(y_n)\geq1/2$, and in
the translated coordinates $u_*(R_0e_d)\geq1/2$.

If $\varrho_*<1$, Lemma~\ref{lem:halfline-classification} for $d=1$ and
\cite[Theorem 1.1(A)]{BerestyckiGrahamHalfSpace} for $d\geq2$ give
$u_*(x',y)=\phi_{\varrho_*}(y)$.  More precisely, we apply that theorem to
the zero extension of $f|_{[0,1]}$ used there; since the relevant solutions take
values in $[0,1]$, this does not change the equation.  If $\varrho_*=1$,
even reflection gives a bounded entire solution; the strong maximum principle,
elliptic regularity, and Lemma~\ref{lem:interior-saturation} yield
$u_*\equiv1=\phi_1$.

Passing to a further subsequence, $d_{\kappa_n}(x_n)\to t\in[0,\infty)$.  In the
translated and rotated
coordinates, the identity $x_n-p_n=d_{\kappa_n}(x_n)e_n$ gives
$x_n\to te_d$.  Local uniform convergence and the preceding classification
then give
\[
 u_n(x_n)\longrightarrow\phi_{\varrho_*}(t).
\]
Also Proposition~\ref{prop:translation} gives
\[
 \Phi_{\kappa_n,\varrho_n}(x_n)
 =\phi_{\varrho_n}(d_{\kappa_n}(x_n))
 \longrightarrow\phi_{\varrho_*}(t).
\]
The two terms in \eqref{eq:bad-point} have the same limit
$\phi_{\varrho_*}(t)$, contradicting the fixed separation $\varepsilon$.
\end{proof}

\section{From half-space stability to large dilations}\label{sec:localization}

Set
\begin{equation}\label{eq:Lambda}
 \Lambda_{\kappa,\varrho}:=
 \lambda\bigl(-\Delta-f'(\Phi_{\kappa,\varrho}),
 D_\kappa,\varrho\bigr).
\end{equation}
The next lemma compares localized Rayleigh quotients under boundary flattening.

\begin{lemma}
\label{lem:mixed-form}
Fix $R>0$.  After rigid motions, suppose that $D_n$ converges locally in
$C^1$ to $\Hh^d$ in the following precise sense.  There are
$z_n\to z$, an open set
$U\supset\overline{B_{R+1}(z)}$, and $C^1$ diffeomorphisms
$F_n:U\to F_n(U)$ such that
\begin{equation}\label{eq:flattening}
 \begin{gathered}
 F_n(U\cap\Hh^d)=F_n(U)\cap D_n,\qquad
 F_n(U\cap\partial\Hh^d)=F_n(U)\cap\partial D_n,\\
 \|F_n-\operatorname{Id}\|_{C^1(U)}\longrightarrow0.
 \end{gathered}
\end{equation}
Let $G_n$ be a connected component of $D_n\cap B_R(z_n)$.  On its
original boundary
\[
 \Gamma_n:=\partial G_n\cap\partial D_n\cap B_R(z_n)
\]
it carries $\varrho_n$ data, and on
$\partial G_n\setminus\Gamma_n$ it carries Dirichlet data.  In particular,
points of $\partial D_n\cap\partial B_R(z_n)$ are assigned to the artificial
boundary.  Suppose that $W\in C_b^{0,\gamma}(\R_+)$ and
\begin{equation}\label{eq:potential-chart-convergence}
 \|V_n\circ F_n-W(y)\|_{L^\infty(U\cap\Hh^d)}\longrightarrow0,
 \qquad \varrho_n\longrightarrow\varrho\in[0,1].
\end{equation}
Denote the resulting mixed principal eigenvalue of $-\Delta+V_n$ on
$G_n$ by $\lambda_n$.  Then
\begin{equation}
 \liminf_{n\to\infty}\lambda_n\geq
 \lambda(-\Delta+W(y),\Hh^d,\varrho).
 \label{eq:mixed-liminf}
\end{equation}
\end{lemma}

\begin{proof}
Pass to a subsequence realizing the finite lower limit.  We may also assume
that either $\varrho_n>0$ for every $n$ or $\varrho_n=0$ for every $n$.
Write
\[
 \Sigma_n:=\partial G_n\setminus\Gamma_n
\]
for the artificial part of the boundary.  When $\varrho_n>0$, let
$H^1_{\Sigma_n}(G_n)$ denote the $H^1(G_n)$-closure of smooth functions on
$\overline{G_n}$ whose traces vanish in a neighborhood of $\Sigma_n$; when
$\varrho_n=0$, use $H_0^1(G_n)$.  These are precisely the mixed
quadratic-form domains: the artificial boundary is Dirichlet, while the
physical boundary is Robin in the first case and Dirichlet in the second.
Denote the corresponding space by $\mathcal H_n$ and take
$w_n\in\mathcal H_n$ with $\|w_n\|_{L^2(G_n)}=1$ and Rayleigh quotient at
most $\lambda_n+n^{-1}$.  Since the potentials are uniformly bounded and the
Robin boundary contribution is nonnegative, these near-minimizers satisfy
\begin{equation}\label{eq:form-bound}
 \int_{G_n}|\nabla w_n|^2
 +\beta_{\varrho_n}
  \int_{\Gamma_n}w_n^2\leq C,
\end{equation}
with the second term omitted for Dirichlet indices.

For all large $n$, $\overline{B_R(z_n)}\subset F_n(U)$.  Extend $w_n$ by
zero from $G_n$ to $D_n\cap F_n(U)$.  By the closure definition of
$H^1_{\Sigma_n}(G_n)$ (and, at Dirichlet indices, of $H_0^1(G_n)$), zero
extension across $\Sigma_n$ is continuous in $H^1$; hence this extension
lies in $H^1(D_n\cap F_n(U))$.  This formulation also removes any ambiguity
at the corner set $\partial D_n\cap\partial B_R(z_n)$, which was assigned to
$\Sigma_n$.  Pull the extension back by $F_n$ and extend again by zero
outside $U$; call the resulting function $\widetilde w_n$ on $\Hh^d$.  Its support is contained in
$F_n^{-1}(\overline{B_R(z_n)})\Subset U$, so
\[
 \widetilde w_n\in H^1(\Hh^d)\quad\text{if }\varrho_n>0,
 \qquad
 \widetilde w_n\in H_0^1(\Hh^d)\quad\text{if }\varrho_n=0.
\]
Accordingly $\widetilde w_n$ belongs to the half-space form domain for
$\varrho_n$.

Let $J_n$, $A_n$, and $J_n^\partial$ be, respectively, the volume
Jacobian, the pullback coefficient matrix for the Dirichlet energy, and the
boundary Jacobian associated with $F_n$.  From \eqref{eq:flattening},
\begin{equation}\label{eq:jacobian-convergence}
 \|J_n-1\|_\infty+\|A_n-I\|_\infty
 +\|J_n^\partial-1\|_\infty\longrightarrow0
\end{equation}
on the fixed neighborhood containing the supports of
$\widetilde w_n$.  Change of variables, the normalization of $w_n$, and
\eqref{eq:potential-chart-convergence} give
\begin{equation}\label{eq:norm-transfer}
 \|\widetilde w_n\|_{L^2(\Hh^d)}^2=1+o(1).
\end{equation}
The near-minimality of $w_n$, together with \eqref{eq:form-bound} and
\eqref{eq:jacobian-convergence} show that
\begin{align}
 &\int_{\Hh^d}\bigl(|\nabla\widetilde w_n|^2
       +W(y)\widetilde w_n^2\bigr)
 +\beta_{\varrho_n}\int_{\partial\Hh^d}\widetilde w_n^2
 \notag\\
 &\qquad\qquad\leq\lambda_n+o(1)
 \label{eq:form-transfer}
\end{align}
when $\varrho_n>0$.  The boundary change of variables contributes at most
\[
 o(1)\,\beta_{\varrho_n}
 \int_{\Gamma_n}w_n^2=o(1)
\]
even when $\beta_{\varrho_n}\to\infty$, by
\eqref{eq:form-bound}.  For Dirichlet indices,
\eqref{eq:form-transfer} holds without the boundary integrals.

Applying the half-space Rayleigh principle to \eqref{eq:norm-transfer} and
\eqref{eq:form-transfer} gives
\[
 \lambda(-\Delta+W(y),\Hh^d,\varrho_n)
 \leq \lambda_n+o(1).
\]
By Lemma~\ref{lem:product}, these half-space eigenvalues are the
corresponding half-line form bottoms.  Lemma
\ref{lem:halfline-form-continuity} gives
\[
 \lambda(-\Delta+W(y),\Hh^d,\varrho_n)
 \longrightarrow\lambda(-\Delta+W(y),\Hh^d,\varrho),
\]
and \eqref{eq:mixed-liminf} follows.
\end{proof}

For the two-sided estimate we need quasimodes with support controlled uniformly
in $\varrho$ as the localization radius grows.  Put
\[
 V_\varrho(y):=-f'(\phi_\varrho(y)),\qquad a_f:=-f'(1),
\]
and define the half-line form
\[
 q_\varrho[h]:=\int_0^\infty\bigl(|h'|^2+V_\varrho h^2\bigr)\,dy
 +\beta_\varrho |h(0)|^2
\]
for $\varrho>0$, with the boundary term omitted and zero trace imposed when
$\varrho=0$.  We also use the notation
\begin{align*}
 Q_\varrho(\Psi)
 &:=\int_{\Hh^d}\bigl(|\nabla\Psi|^2+V_\varrho(y)\Psi^2\bigr)
   +\beta_\varrho\int_{\partial\Hh^d}\Psi^2,\\
 Q_{\kappa,\varrho}(\zeta)
 &:=\int_{D_\kappa}\bigl(|\nabla\zeta|^2
        -f'(\Phi_{\kappa,\varrho})\zeta^2\bigr)
   +\beta_\varrho\int_{\partial D_\kappa}\zeta^2,
\end{align*}
with the boundary terms omitted and the form domains replaced by zero-trace
spaces when $\varrho=0$.  The endpoint assumptions,
\eqref{eq:uniform-tail}, and the H\"older continuity of $f'$ give constants
$C,c>0$ such that
\begin{equation}\label{eq:uniform-potential-tail}
 \sup_{\varrho\in[0,1]}|V_\varrho(y)-a_f|\leq Ce^{-cy}
 \qquad(y\geq0).
\end{equation}

\begin{lemma}
\label{lem:uniform-quasimodes}
There are $R_0,C>0$ such that, for every $R\geq R_0$ and
$\varrho\in[0,1]$, one can find an $L^2$-normalized function
$\Psi_{\varrho,R}$ in the $\varrho$ form domain on $\Hh^d$ satisfying
\begin{equation}\label{eq:uniform-quasimode}
 \operatorname{supp}\Psi_{\varrho,R}\subset B_{3R}(0)\cap\overline{\Hh^d},
 \qquad
 Q_\varrho(\Psi_{\varrho,R})\leq\mu_\varrho+\frac{C}{R}.
\end{equation}
At the Dirichlet endpoint the quasimode has zero trace.
\end{lemma}

\begin{proof}
On the half-line, set
$\delta_\varrho:=a_f-\mu_\varrho\geq0$ and fix
$\eta\in C_c^\infty(0,1)$ with $\|\eta\|_2=1$.  We distinguish the two cases
$\delta_\varrho\leq2R^{-1}$ and $\delta_\varrho>2R^{-1}$.  In the first a
distant cutoff is sufficient; in the second we truncate the ground state.

If $\delta_\varrho\leq2R^{-1}$, choose $Y_R=C_0\log R$, where $C_0$ is
large enough that \eqref{eq:uniform-potential-tail} is at most $R^{-2}$ on
$[Y_R,\infty)$.  For large $R$ one has $Y_R<R$, and
\[
 h_{\varrho,R}(y):=R^{-1/2}\eta\!\left(\frac{y-Y_R}{R}\right)
\]
has zero trace, unit norm, and support in $[0,2R]$.  It is admissible for
every boundary parameter, and
\[
 q_\varrho[h_{\varrho,R}]
 \leq a_f+CR^{-2}
 \leq\mu_\varrho+2R^{-1}+CR^{-2}.
\]

When $\delta_\varrho>2R^{-1}$ the bottom of the spectrum is separated from
$a_f$, so an $L^2$ ground state is available.  Let $z_\varrho$ denote its positive $L^2$-normalized ground state.  Its $H^1$ norm is bounded independently of
$\varrho$: test the eigenvalue equation by $z_\varrho$, use
$\mu_\varrho\leq a_f$, the uniform bound on $V_\varrho$, and the
nonnegative Robin boundary term.  If
\[
 Y_\delta=C_1\log(2+\delta^{-1}),
\]
then \eqref{eq:uniform-potential-tail} and a larger fixed $C_1$ give
$V_\varrho-a_f\geq-\delta/4$ on $[Y_\delta,\infty)$.  With
$\delta=\delta_\varrho$ and
$k_\varrho:=\sqrt{3\delta_\varrho/4}$, the eigenvalue equation gives
\[
 z_\varrho''=(V_\varrho-a_f+\delta_\varrho)z_\varrho
 \geq k_\varrho^2z_\varrho
 \qquad\text{on }[Y_{\delta_\varrho},\infty).
\]
Since the potential is bounded, $z_\varrho\in H^2(\R_+)$ and
$z_\varrho,z_\varrho'\to0$ at infinity.  Set
$F_\varrho=z_\varrho'+k_\varrho z_\varrho$.  The preceding inequality yields
$F_\varrho'\geq k_\varrho F_\varrho$.  If $F_\varrho(y_0)>0$ at one point,
then $F_\varrho(y)\geq F_\varrho(y_0)e^{k_\varrho(y-y_0)}$ for every
$y\geq y_0$, contradicting $F_\varrho(y)\to0$.  The inequality therefore
forces $z_\varrho'+k_\varrho z_\varrho\leq0$, and integration gives
\begin{equation}\label{eq:ground-state-tail}
 |z_\varrho(y)|\leq
 Ce^{-\sqrt{3\delta_\varrho/4}\,(y-Y_{\delta_\varrho})}
 \qquad(y\geq Y_{\delta_\varrho}).
\end{equation}
The prefactor is uniform by the preceding $H^1$ bound and the one-dimensional
trace inequality.  Since $\delta_\varrho>2R^{-1}$,
$Y_{\delta_\varrho}=O(\log R)$ and
$\sqrt{\delta_\varrho}\,R\geq\sqrt{2R}$.  For all sufficiently large
$R$, uniformly in this regime, $R-Y_{\delta_\varrho}\geq R/2$, and
\begin{equation}\label{eq:ground-state-tail-mass}
 \int_R^\infty z_\varrho^2
 \leq C\delta_\varrho^{-1/2}
 e^{-c\sqrt{\delta_\varrho}(R-Y_{\delta_\varrho})}
 \leq CR^{1/2}e^{-c\sqrt R}=o(1)
\end{equation}
uniformly in this case.

Let $\theta_R$ equal one on $[0,R]$, vanish on $[2R,\infty)$, and satisfy
$|\theta_R'|\leq C/R$.  Testing the eigenvalue equation by
$\theta_R^2z_\varrho$ gives the exact IMS identity
\[
 q_\varrho[\theta_Rz_\varrho]
 =\mu_\varrho\|\theta_Rz_\varrho\|_2^2
  +\int_0^\infty|\theta_R'|^2z_\varrho^2.
\]
By \eqref{eq:ground-state-tail-mass}, the denominator is at least $1/2$ for
all large $R$, uniformly in $\varrho$.  After normalization this produces a
half-line test function supported in $[0,2R]$ with Rayleigh quotient at most
$\mu_\varrho+CR^{-2}$.  The first regime already had an $O(R^{-1})$ error, so the two cases give a
uniform $C/R$ bound on the half-line.

For $d\geq2$, multiply the half-line function by an $L^2$-normalized
tangential cutoff of radius $R$ and Dirichlet energy $O(R^{-2})$; for $d=1$
there is no tangential factor.  Lemma~\ref{lem:product} gives
\eqref{eq:uniform-quasimode}.
\end{proof}

\begin{lemma}
\label{lem:growing-fermi}
Assume $\partial\Omega\neq\varnothing$.  There are $c_0,C,c>0$ such that,
whenever $1\leq R\leq c_0\kappa$, the following statements hold uniformly in
$\varrho\in[0,1]$.
\begin{enumerate}[label=(\roman*)]
\item If $G$ is a connected component of $D_\kappa\cap B_R(z)$, with the
original boundary carrying $\varrho$ data and the artificial boundary carrying
Dirichlet data, then
\begin{equation}\label{eq:growing-fermi-lower}
 \lambda\bigl(-\Delta-f'(\Phi_{\kappa,\varrho}),G,\varrho\bigr)
 \geq\mu_\varrho-C\left(\frac{R}{\kappa}+e^{-cR}\right).
\end{equation}
\item For every $p\in\partial\Omega$ and every half-space form function
$\Psi$ supported in $B_{3R}\cap\overline{\Hh^d}$ with, for some $C_1\geq1$,
\[
 \|\Psi\|_2=1,
 \qquad
 \int_{\Hh^d}|\nabla\Psi|^2
 +\beta_\varrho\int_{\partial\Hh^d}\Psi^2\leq C_1,
\]
where the boundary term is omitted at $\varrho=0$, there is an admissible
function $T_{\kappa,p}\Psi$ on $D_\kappa$ such that
\begin{equation}\label{eq:growing-fermi-upper}
 \frac{Q_{\kappa,\varrho}(T_{\kappa,p}\Psi)}
      {\|T_{\kappa,p}\Psi\|_2^2}
 \leq Q_\varrho(\Psi)+CC_1\frac{R}{\kappa}.
\end{equation}
\end{enumerate}
\end{lemma}

\begin{proof}
Fix $p\in\partial\Omega$, translate $\kappa p$ to the origin, and rotate the
inward normal to $e_d$.  By the uniform $C^{2,\gamma}$ geometry of $\Omega$, we may shrink $c_0$
so that, whenever $R\leq c_0\kappa$, the normal coordinates around
$\kappa p$ are injective on the portion of the tubular neighborhood used
below.  Their coordinate domain
contains $B_{6R}\cap\Hh^d$, and their image contains
$D_\kappa\cap B_{5R}(0)$.  Write
\[
 F_{\kappa,p}(x',y)=b_{\kappa,p}(x')+y\,n_{\kappa,p}(x').
\]
Scaling the fixed boundary charts gives, for $|(x',y)|\leq6R$,
\begin{equation}\label{eq:growing-fermi-geometry}
 d_\kappa(F_{\kappa,p}(x',y))=y,\qquad
 \|DF_{\kappa,p}-I\|_\infty
 +\|J_{\kappa,p}-1\|_\infty
 +\|J^\partial_{\kappa,p}-1\|_\infty
 \leq C\frac{R}{\kappa}.
\end{equation}
After decreasing $c_0$ further, these maps and their inverses are uniformly
bi-Lipschitz.  If
\[
 A_{\kappa,p}:=J_{\kappa,p}
 DF_{\kappa,p}^{-1}DF_{\kappa,p}^{-T},
\]
then the same estimate also gives
\begin{equation}\label{eq:growing-fermi-metric}
 \|A_{\kappa,p}-I\|_\infty\leq C\frac{R}{\kappa}.
\end{equation}
Moreover,
\begin{equation}\label{eq:exact-potential-transfer}
 -f'(\Phi_{\kappa,\varrho})\circ F_{\kappa,p}=V_\varrho(y).
\end{equation}

For a form function $\Psi$ supported in $B_{3R}\cap\overline{\Hh^d}$, set
$T_{\kappa,p}\Psi=\Psi\circ F_{\kappa,p}^{-1}$ on the image of the chart and
extend it by zero outside that image.  Its support stays a fixed distance from
the artificial edge of the chart, so this zero extension is admissible.  By
change of variables,
\begin{align}
 \|T_{\kappa,p}\Psi\|_2^2
 &=\int_{\Hh^d}\Psi^2 J_{\kappa,p},
 \label{eq:growing-fermi-norm-transfer}\\
 Q_{\kappa,\varrho}(T_{\kappa,p}\Psi)
 &=\int_{\Hh^d}\nabla\Psi^T A_{\kappa,p}\nabla\Psi
   +\int_{\Hh^d}V_\varrho(y)\Psi^2J_{\kappa,p}\notag\\
 &\quad+\beta_\varrho\int_{\partial\Hh^d}
       \Psi^2J^\partial_{\kappa,p},
 \label{eq:growing-fermi-form-transfer}
\end{align}
with the boundary term omitted at $\varrho=0$.  Hence
\begin{align}
 \left|\|T_{\kappa,p}\Psi\|_2^2-\|\Psi\|_2^2\right|
 &\leq C\frac{R}{\kappa}\|\Psi\|_2^2,
 \label{eq:growing-fermi-norm-error}\\
 |Q_{\kappa,\varrho}(T_{\kappa,p}\Psi)-Q_\varrho(\Psi)|
 &\leq C\frac{R}{\kappa}
 \left(\|\Psi\|_{H^1}^2
 +\beta_\varrho\|\Psi\|_{L^2(\partial\Hh^d)}^2\right).
 \label{eq:growing-fermi-form-error}
\end{align}
Here the $L^2$ term in the $H^1$ norm controls the potential contribution,
since $V_\varrho$ is uniformly bounded.  The boundary error is
$O(R/\kappa)$ times the Robin energy itself; this is why the constants remain
uniform when $\beta_\varrho\to\infty$.  The same estimates, with the roles of
the chart and the half-space reversed, hold for pullbacks of functions
supported in $D_\kappa\cap B_{4R}(0)$.

Under the assumptions in (ii), the form bound and the uniform bound on
$V_\varrho$ imply $\|\Psi\|_{H^1}^2+\beta_\varrho
\|\Psi\|_{L^2(\partial\Hh^d)}^2\leq C(1+C_1)$.  Equations
\eqref{eq:growing-fermi-norm-error}--\eqref{eq:growing-fermi-form-error}
therefore give
\[
 \frac{Q_{\kappa,\varrho}(T_{\kappa,p}\Psi)}
      {\|T_{\kappa,p}\Psi\|_2^2}
 \leq Q_\varrho(\Psi)+CC_1\frac{R}{\kappa},
\]
after decreasing $c_0$ so that $CR/\kappa\leq1/2$.  This proves (ii).

For (i), suppose first that $\dist(G,\partial D_\kappa)>R$.  Then
$d_\kappa\geq R$ on $G$, and \eqref{eq:uniform-potential-tail} gives
\[
 \lambda\bigl(-\Delta-f'(\Phi_{\kappa,\varrho}),G,\varrho\bigr)
 \geq a_f-Ce^{-cR}\geq\mu_\varrho-Ce^{-cR}.
\]
Assume now that $\dist(G,\partial D_\kappa)\leq R$.  Choose
$x\in G$ with $d_\kappa(x)\leq R+1$ and a nearest boundary point $P$.
Since $G\subset B_R(z)$ and $R\geq1$, every point of $G$ lies in
$B_{4R}(P)$.  Recenter the preceding Fermi construction at $P$.  By the
choice of $c_0$, the whole component $G$ lies in the image of the chart, and
its pullback is contained in $B_{5R}\cap\Hh^d$.

Let $\lambda_G$ denote the mixed form bottom on $G$, and take an
$L^2$-normalized $\varepsilon$-minimizer $w$.  Extend $w$ by zero across the
artificial part of $\partial G$ and pull it back to the half-space; denote the
result by $\widetilde w$.  This zero extension lies in $H^1$ because members
of the mixed form domain have zero trace on the artificial boundary.  Choose
$K>1+\sup_\varrho\mu_\varrho$.  If $\lambda_G\geq K$, then
\eqref{eq:growing-fermi-lower} is immediate after increasing $C$.  Otherwise,
for $\varepsilon\leq1$,
\[
 \int_G|\nabla w|^2
 +\beta_\varrho\int_{\partial G\cap\partial D_\kappa}w^2
 \leq K+1+\sup_{\varrho}\|V_\varrho\|_\infty,
\]
with the boundary term omitted at $\varrho=0$.  Applying the inverse versions
of \eqref{eq:growing-fermi-norm-error} and
\eqref{eq:growing-fermi-form-error} gives
\[
 \|\widetilde w\|_2^2=1+O(R/\kappa),\qquad
 Q_\varrho(\widetilde w)
 \leq\lambda_G+\varepsilon+C\frac{R}{\kappa}.
\]
Since $\lambda_G$ is bounded above by $K$ in this case, division by
$\|\widetilde w\|_2^2$ changes the right-hand side by only
$O(R/\kappa)$.  The half-space Rayleigh principle therefore yields
\[
 \mu_\varrho\leq\lambda_G+\varepsilon+C\frac{R}{\kappa}.
\]
Letting $\varepsilon\downarrow0$ proves (i).
\end{proof}

\begin{proposition}\label{prop:large-gap}
For $\mu_\varrho$ from \eqref{eq:mu-rho},
\begin{equation}\label{eq:large-gap-liminf}
 \liminf_{\kappa\to\infty}\inf_{\varrho\in[0,1]}
 \bigl(\Lambda_{\kappa,\varrho}-\mu_\varrho\bigr)\geq0.
\end{equation}
As a consequence, with $\mu_*$ from \eqref{eq:uniform-gap}, there is
$\kappa_1(f,\Omega)$ such that
\begin{equation}\label{eq:large-gap}
 \Lambda_{\kappa,\varrho}\geq\frac{\mu_*}{2}
 \qquad(\kappa\geq\kappa_1,\ \varrho\in[0,1]).
\end{equation}
\end{proposition}

\begin{proof}
The whole-space case is immediate from
$\Lambda_{\kappa,\varrho}=|f'(1)|\geq\mu_\varrho$, so take
$\partial\Omega\neq\varnothing$.  If \eqref{eq:large-gap-liminf} were false,
then for some fixed $\varepsilon>0$ (reduced once if needed) there would be
\[
 \kappa_n\to\infty,\qquad \varrho_n\in[0,1],
\]
such that
\begin{equation}\label{eq:fixed-R-contradiction}
 \Lambda_{\kappa_n,\varrho_n}
 \leq \mu_{\varrho_n}-5\varepsilon.
\end{equation}
Passing to a subsequence, $\varrho_n\to\varrho_*\in[0,1]$; Proposition
\ref{prop:uniform-gap} gives $\mu_{\varrho_n}\to\mu_{\varrho_*}$.

Choose once and for all $R>0$ so large that
\[
 \lambda_1(B_1)R^{-2}<\varepsilon.
\]
The Robin--Lieb inequality \eqref{eq:lieb} implies
\[
 \inf_{z\in\R^d}
 \lambda\bigl(-\Delta-f'(\Phi_{\kappa_n,\varrho_n}),
 D_{\kappa_n}\mid B_R(z),\varrho_n\bigr)
 \leq \mu_{\varrho_n}-4\varepsilon.
\]
Choose $z_n$ so that the localized form bottom is at most
$\mu_{\varrho_n}-3\varepsilon$.  If
$D_{\kappa_n}\cap B_R(z_n)$ is disconnected, its form bottom is the infimum
of the component bottoms.  We may choose a connected component $G_n$ whose
mixed Robin--Dirichlet principal eigenvalue $\lambda_n$ satisfies
\begin{equation}\label{eq:fixed-R-local-low}
 \lambda_n\leq\mu_{\varrho_n}-2\varepsilon.
\end{equation}
Set
\[
 m_n:=\dist(G_n,\partial D_{\kappa_n}).
\]
After passing to a subsequence, either $m_n\to\infty$ or $(m_n)$ remains
bounded.  We treat the two cases separately.
In the first case $G_n$ eventually misses the original boundary, so its boundary condition is purely Dirichlet,
and $d_{\kappa_n}(x)\geq m_n$ on $G_n$.  The uniform profile tail \eqref{eq:uniform-tail}, together with the uniform
continuity of $f'$, gives
\[
 \sup_{x\in G_n}
 \bigl|-f'(\Phi_{\kappa_n,\varrho_n}(x))+f'(1)\bigr|\longrightarrow0.
\]
Writing $a_f:=-f'(1)$ and using the nonnegativity of the Dirichlet energy
gives
\[
 \liminf_{n\to\infty}\lambda_n\geq a_f.
\]
On the other hand, test functions translated to infinity in the half-line
problem give $\mu_\varrho\leq a_f$ for every $\varrho$; cf. the proof of
Proposition~\ref{prop:uniform-gap}.  By \eqref{eq:fixed-R-local-low} and
$\mu_{\varrho_n}\to\mu_{\varrho_*}$,
\[
 \limsup_{n\to\infty}\lambda_n
 \leq \mu_{\varrho_*}-2\varepsilon
 \leq a_f-2\varepsilon,
\]
contradicting the preceding lower bound.

After passing to a subsequence, take $M$ with $m_n\leq M$.  Choose
$x_n\in G_n$ with $d_{\kappa_n}(x_n)\leq M+o(1)$ and let
$p_n\in\partial D_{\kappa_n}$ be a nearest boundary point.  Since
$x_n\in B_R(z_n)$,
\[
 |z_n-p_n|\leq R+M+o(1).
\]
Translate $p_n$ to the origin and rotate the inward normal there to $e_d$.
Taking a further subsequence if needed, the transformed centers converge,
say $z_n\to z$.  The uniform $C^{2,\gamma}$ geometry of $\Omega$, after dilation, supplies Fermi
flattenings $F_n$ on a fixed neighborhood of
$\overline{B_{R+1}(z)}$ such that
\[
 F_n(\Hh^d)=D_{\kappa_n}\quad\hbox{locally},\qquad
 \|F_n-\operatorname{Id}\|_{C^1}\longrightarrow0,
\]
and
\begin{equation}\label{eq:fixed-R-distance-coordinate}
 d_{\kappa_n}(F_n(x',y))=y.
\end{equation}
In these coordinates,
\[
 -f'(\Phi_{\kappa_n,\varrho_n})\circ F_n
 =-f'(\phi_{\varrho_n}(y))
 \longrightarrow -f'(\phi_{\varrho_*}(y))
\]
uniformly on that fixed neighborhood, by Proposition~\ref{prop:translation}
and the uniform continuity of $f'$.  Applying Lemma~\ref{lem:mixed-form} to
the chosen components $G_n$ yields
\[
 \liminf_{n\to\infty}\lambda_n
 \geq \lambda\bigl(-\Delta-f'(\phi_{\varrho_*}(y)),
                    \Hh^d,\varrho_*\bigr)
 =\mu_{\varrho_*},
\]
where the last equality is Lemma~\ref{lem:product}.  Meanwhile,
\eqref{eq:fixed-R-local-low} and continuity of $\mu_\varrho$ give
\[
 \limsup_{n\to\infty}\lambda_n
 \leq\mu_{\varrho_*}-2\varepsilon,
\]
again impossible.  This contradiction gives \eqref{eq:large-gap-liminf};
the uniform positive lower bound $\mu_\varrho\geq\mu_*>0$ then yields
\eqref{eq:large-gap}.

\end{proof}

\begin{remark}
\label{rem:minimal-spectral-input}
Proposition~\ref{prop:large-gap} is the only spectral result from this
section used in the proof of Theorem~\ref{thm:main}.  The growing-radius
argument is only for the two-sided estimate.
\end{remark}

\begin{theorem}
\label{thm:spectral-convergence}
Assume $\partial\Omega\neq\varnothing$.  There are $C>0$ and
$\kappa_{\rm sp}>0$ such that
\begin{equation}\label{eq:spectral-convergence}
 \sup_{\varrho\in[0,1]}
 \bigl|\Lambda_{\kappa,\varrho}-\mu_\varrho\bigr|
 \leq C\kappa^{-1/2}
 \qquad(\kappa\geq\kappa_{\rm sp}).
\end{equation}
\end{theorem}

\begin{proof}
Set $R=\sqrt\kappa$.  After increasing the lower threshold for $\kappa$, this
choice satisfies both $R\geq R_0$ and $R\leq c_0\kappa$.  The Robin--Lieb inequality and
Lemma~\ref{lem:growing-fermi}(i), applied to every connected component of
every localization, give
\begin{equation}\label{eq:quantitative-large-lower}
 \Lambda_{\kappa,\varrho}
 \geq\mu_\varrho-C\left(\frac{R}{\kappa}+e^{-cR}+R^{-2}\right)
 \geq\mu_\varrho-C\kappa^{-1/2}.
\end{equation}
For the opposite inequality, fix $p\in\partial\Omega$ and take the quasimode
$\Psi_{\varrho,R}$ from Lemma~\ref{lem:uniform-quasimodes}.  Since
$\|\Psi_{\varrho,R}\|_2=1$, \eqref{eq:uniform-quasimode} and the uniform
bound on $V_\varrho$ give a constant $C_1$, independent of $R$ and
$\varrho$, such that
\[
 \int_{\Hh^d}|\nabla\Psi_{\varrho,R}|^2
 +\beta_\varrho\int_{\partial\Hh^d}\Psi_{\varrho,R}^2\le C_1,
\]
with the boundary term omitted at $\varrho=0$.  Lemma~\ref{lem:growing-fermi}(ii), followed by the Rayleigh principle, gives
\[
 \Lambda_{\kappa,\varrho}
 \leq Q_\varrho(\Psi_{\varrho,R})+C\frac{R}{\kappa}
 \leq\mu_\varrho+C\left(R^{-1}+\frac{R}{\kappa}\right)
 \leq\mu_\varrho+C\kappa^{-1/2}.
\]
Combining the two bounds proves \eqref{eq:spectral-convergence}, uniformly
in $\varrho$.
\end{proof}

The exponent $1/2$ is not asserted to be sharp for a fixed parameter.  It is
the rate produced by a construction that remains uniform when
$\mu_\varrho$ approaches the bottom of the essential spectrum.  The condition
$\partial\Omega\ne\varnothing$ is needed for the two-sided limit: for
$\Omega=\R^d$, $\Lambda_{\kappa,\varrho}=|f'(1)|$ whereas $\mu_\varrho$ in
general depends on $\varrho$.

By Proposition~\ref{prop:coalescence} and
\eqref{eq:potential-lipschitz}, one also has
\[
 \sup_{\varrho\in[0,1]}\sup_{u\in\cU_{\kappa,\varrho}}
 \left|\lambda\bigl(-\Delta-f'(u),D_\kappa,\varrho\bigr)-\mu_\varrho\right|
 \longrightarrow0.
\]
No rate is asserted here; the quantitative estimate under the stronger
hypotheses of Theorem~\ref{thm:curvature-correction} is given in
Corollary~\ref{cor:quantitative-solution-stability}.  As in
Proposition~\ref{prop:coalescence}, a supremum over an empty solution set is
understood as zero.

\section{Uniqueness and stability of the nonlinear states}\label{sec:proof-main}

Existence needs no boundary-layer argument: an interior Dirichlet test
function is admissible for every boundary parameter.

\begin{lemma}\label{lem:existence}
There is $\kappa_{\rm ex}(f,\Omega)$ such that
$\cU_{\kappa,\varrho}$ is nonempty for every
$\kappa\geq\kappa_{\rm ex}$ and $\varrho\in[0,1]$.
\end{lemma}

\begin{proof}
Fix $x_0\in\Omega$ and $r>0$ with $B_r(x_0)\Subset\Omega$, and take $R>0$
so large that
\[
 \lambda_1(B_R)<f'(0).
\]
For every $\kappa\geq R/r$, the dilated domain $D_\kappa$ contains the ball
$B_R(\kappa x_0)$.  Let $\psi_R$ be an $L^2$-normalized first Dirichlet
eigenfunction on this ball, extended by zero to $D_\kappa$.  Since the ball is
compactly contained in $D_\kappa$, this extension belongs to
$H_0^1(D_\kappa)$ and is an admissible form test function for every
$\varrho\in[0,1]$; for $\varrho>0$ its trace on $\partial D_\kappa$ is zero,
so the Robin boundary term also vanishes.  The form characterization
\eqref{eq:form-bottom-definition} gives simultaneously for all boundary
parameters
\begin{equation}\label{eq:existence-eigenvalue}
 \lambda(-\Delta,D_\kappa,\varrho)
 \leq \lambda_1(B_R)<f'(0).
\end{equation}
Here the potential is zero, so the form bottom agrees with the generalized
principal eigenvalue used in
\cite[Theorem~1.5]{BerestyckiGrahamStrongKPP}.  Since $f$ is a positive
reaction, that theorem applied to \eqref{eq:existence-eigenvalue} yields a
bounded positive solution of \eqref{eq:rho-problem} for every
$\varrho\in[0,1]$, including the Neumann endpoint.  We may take
$\kappa_{\rm ex}=R/r$, independently of $\varrho$.
\end{proof}

\begin{proof}[Proof of Theorem~\ref{thm:main}]
Let $\mu_*>0$ be given by Proposition~\ref{prop:uniform-gap}.  Uniform
continuity of $f'$ on $[0,1]$ provides $\delta>0$ such that
\begin{equation}\label{eq:delta-choice}
 |a-b|\leq\delta
 \quad\Longrightarrow\quad
 |f'(a)-f'(b)|\leq\frac{\mu_*}{4}.
\end{equation}
By Propositions~\ref{prop:coalescence} and \ref{prop:large-gap}, there is
$\kappa_0\geq\max\{\kappa_{\rm ex},\kappa_1\}$ such that, for every
$\kappa\geq\kappa_0$ and $\varrho\in[0,1]$,
\begin{equation}\label{eq:two-uniform-estimates}
 \sup_{u\in\cU_{\kappa,\varrho}}
 \|u-\Phi_{\kappa,\varrho}\|_\infty\leq\delta,
 \qquad
 \Lambda_{\kappa,\varrho}\geq\frac{\mu_*}{2}.
\end{equation}
Let $u,v\in\cU_{\kappa,\varrho}$ and write $w=u-v$.  At each point, the
fundamental theorem of calculus gives
\[
 f(u)-f(v)=q(x)w,\qquad
 q(x):=\int_0^1 f'\bigl(v(x)+t(u(x)-v(x))\bigr)\,dt .
\]
Every argument of $f'$ in this integral lies within $\delta$ of
$\Phi_{\kappa,\varrho}(x)$ by the first estimate in
\eqref{eq:two-uniform-estimates}.  Using \eqref{eq:delta-choice} and
\eqref{eq:potential-lipschitz} with the second estimate in
\eqref{eq:two-uniform-estimates} gives
\begin{equation}\label{eq:mean-gap}
 \lambda(-\Delta-q,D_\kappa,\varrho)
 \geq\frac{\mu_*}{4}>0.
\end{equation}
Moreover,
\[
 (-\Delta-q)w=0\quad\text{in }D_\kappa,
 \qquad \cN_\varrho w=0\quad\text{on }\partial D_\kappa.
\]
Since $0<u,v\leq1$ and $f'$ is continuous on $[0,1]$, one has
$q\in L^\infty(D_\kappa)$, while $w\in H^1_{\rm loc}(\overline{D_\kappa})\cap
L^\infty(D_\kappa)$.  Apply Lemma~\ref{lem:bounded-weak-maximum} to $w$ with potential $V=-q$.
Equation \eqref{eq:mean-gap} and the homogeneous boundary condition give
$w\geq0$; applying the same lemma to $-w$ gives $w\leq0$.  Hence $u=v$.
Existence was established in Lemma~\ref{lem:existence}.

Finally, for $0\leq\varrho<1$ we have
$\alpha=\varrho/(1-\varrho)$, and \eqref{eq:compact-param} becomes
$u+\alpha\partial_\nu u=0$.  The value $\varrho=0$ corresponds to
$\alpha=0$, while $\varrho=1$ is the adjoined Neumann endpoint
$\alpha=+\infty$.  Hence the threshold $\kappa_0$ found above is the same
throughout the Dirichlet--Robin--Neumann family.
\end{proof}

\subsection{Compact reaction families}

\begin{proof}[Proof of Theorem~\ref{thm:reaction-family}]
Only the restrictions $f|_{[0,1]}$ enter the argument, since every positive
state takes values in $(0,1]$.  The quantities in
\eqref{eq:family-basic-data} are uniform on $\mathfrak F$; in particular, for
each $s\in(0,1)$,
\[
 \inf_{\substack{f\in\mathfrak F\\0<r\le s}}\frac{f(r)}r>0.
\]
Thus Lemma~\ref{lem:interior-saturation} is uniform in $f$.

Given sequences $f_n\in\mathfrak F$ and $\varrho_n\in[0,1]$, compactness allows us
to pass to a subsequence with $f_n\to f_*$ in $C^1([0,1])$ and
$\varrho_n\to\varrho_*$.  The compactness argument of
Lemma~\ref{lem:joint-compactness} is unchanged.  At the Dirichlet
endpoint the rescaled equation is
$-\Delta v_n=\beta_n^{-2}f_n(v_n)$; the common $C^1$ bound makes its right-hand
side tend uniformly to zero, so the same Robin Liouville argument forces the
boundary traces to vanish.  Proposition~\ref{prop:family-halfline} then gives
\[
 \phi^{f_n}_{\varrho_n}\longrightarrow
 \phi^{f_*}_{\varrho_*}\qquad\hbox{uniformly on }\R_+,
\]
and the proof of Proposition~\ref{prop:coalescence} is uniform in
$(f,\varrho)$.

The potentials converge on compact boundary charts, while
\eqref{eq:family-common-exponential-tail} controls the far field.  The proof of
Proposition~\ref{prop:large-gap}, together with
\eqref{eq:family-uniform-gap}, therefore yields the lower bound
$\widehat\mu/2$ beyond one common dilation threshold.  Existence follows from
the same Dirichlet test function used in Lemma~\ref{lem:existence}, with
$\lambda_1(B_R)<a_0$.  Finally, the common H\"older bound in
\eqref{eq:compact-reaction-family} gives a uniform modulus of continuity for
$f'$, so the maximum-principle argument in Theorem~\ref{thm:main} is uniform as
well.
\end{proof}

\section{The parameter-uniform curvature correction}
\label{sec:curvature}

We next derive the first curvature correction.  The additional point needed for
uniformity in $\varrho$ is continuity of the half-line resolvent at the
Dirichlet endpoint.

\begin{lemma}
\label{lem:halfline-resolvent-continuity}
Let $\varrho_n\to\varrho$ in $[0,1]$, let $V_n,V\in L^\infty(\R_+)$ be real
with $V_n\to V$ in $L^\infty$, and let $g_n\to g$ in $L^2(\R_+)$.  Denote by
$\mathfrak a_n$ and $\mathfrak a$ the forms of
$-\partial_y^2+V_n$ and $-\partial_y^2+V$, respectively, with the indicated
$\varrho_n$ and $\varrho$ boundary conditions.  Suppose that, for some
$m>0$,
\[
 \mathfrak a_n[h]\geq m\|h\|_2^2,
 \qquad \mathfrak a[h]\geq m\|h\|_2^2
\]
on the corresponding form domains.  If
\[
 z_n=(-\partial_y^2+V_n)^{-1}g_n,
 \qquad z=(-\partial_y^2+V)^{-1}g,
\]
then $z_n\to z$ in $H^1(\R_+)$.  In fact the convergence is in $H^2$ and
hence in $C_b^1$.
\end{lemma}

\begin{proof}
Testing the equations by $z_n$ gives uniform $L^2$ bounds.  Since the
potentials are uniformly bounded and the Robin boundary terms are
nonnegative, the same identities give
\begin{equation}\label{eq:resolvent-uniform-form-bound}
 \sup_n\bigl(\|z_n\|_{H^1}^2
   +\beta_{\varrho_n}|z_n(0)|^2\bigr)<\infty,
\end{equation}
where the last term is omitted at Dirichlet indices.

If $\varrho>0$, then
$\beta_{\varrho_n}\to\beta_\varrho<\infty$.  With $h_n=z_n-z$, subtracting
the two variational equations gives
\begin{align*}
 \mathfrak a_n[h_n]
 &=\langle g_n-g,h_n\rangle
   +\int_0^\infty(V-V_n)zh_n\,dy\\
 &\quad +(\beta_\varrho-\beta_{\varrho_n})z(0)h_n(0).
\end{align*}
The assumed $L^2$ coercivity, the uniform $L^\infty$ bound on the potentials,
and nonnegativity of the boundary term imply a uniform $H^1$ coercivity
estimate.  Indeed, if $M=\sup_n\|V_n\|_\infty$, then
\[
 \|h\|_{H^1}^2
 \leq \bigl(1+(1+M)/m\bigr)\mathfrak a_n[h]
\]
for every function in the $n$th form domain.  The trace inequality and the
preceding identity therefore show that $h_n\to0$ in $H^1$.

At $\varrho=0$ there are two possible types of indices.  Along any
subsequence with $\varrho_n=0$, both $z_n$ and $z$ lie in $H_0^1(\R_+)$;
subtracting the two variational equations and using the same uniform $H^1$
coercivity gives $z_n\to z$ directly.  It remains to treat a subsequence with
$\varrho_n>0$.  For such a subsequence
$\beta_{\varrho_n}\to\infty$.  Put $b_n=z_n(0)$.  By
\eqref{eq:resolvent-uniform-form-bound}, $b_n\to0$.  The trace-corrected
functions
\[
 \widetilde z_n:=z_n-b_ne^{-y}
\]
belong to $H_0^1(\R_+)$.  Set $h_n=\widetilde z_n-z$ and subtract the two
weak equations using $h_n$ as test function.  Since $h_n(0)=0$, the Robin
term disappears and
\begin{align*}
 \int_0^\infty\bigl(|h_n'|^2+V_nh_n^2\bigr)\,dy
 &=\langle g_n-g,h_n\rangle
   +\int_0^\infty(V-V_n)zh_n\,dy\\
 &\quad-b_n\int_0^\infty
       \bigl((e^{-y})'h_n'+V_ne^{-y}h_n\bigr)\,dy.
\end{align*}
On zero-trace functions the left side is $\mathfrak a_n[h_n]$ and is
uniformly $H^1$ coercive, as above.  Every term on the right is
$o(1)\|h_n\|_{H^1}$, so $h_n\to0$ in $H^1$.  Since
$b_ne^{-y}\to0$ in $H^1$, also $z_n\to z$ in $H^1$.

The second derivatives are then read directly from the ODE:
\[
 z_n''=V_nz_n-g_n,
 \qquad z''=Vz-g,
\]
and the convergence already obtained gives $z_n''\to z''$ in $L^2$.  Thus
$z_n\to z$ in $H^2$, and the one-dimensional Sobolev embedding yields convergence
in $C_b^1$.
\end{proof}

\begin{lemma}
\label{lem:curvature-response}
For every $\varrho\in[0,1]$, problem
\eqref{eq:intro-curvature-response} has a unique solution
$\chi_\varrho\in H^1(\R_+)$.  In addition,
\begin{equation}\label{eq:response-sign-continuity}
 \chi_\varrho\leq0,\qquad
 [0,1]\ni\varrho\longmapsto\chi_\varrho
 \text{ is continuous into }H^1(\R_+)\cap C_b^1(\R_+),
\end{equation}
and there are constants $C,c>0$, independent of $\varrho$, such that
\begin{equation}\label{eq:response-exponential}
 |\phi_\varrho'(y)|+
 \sum_{j=0}^2|\chi_\varrho^{(j)}(y)|
 \leq Ce^{-cy}\qquad(y\geq0).
\end{equation}
At the Neumann endpoint this gives $\chi_1=0$.
\end{lemma}

\begin{proof}
Let $L_\varrho=-\partial_y^2-f'(\phi_\varrho)$ with the
$\varrho$ boundary condition, and let $\mathfrak q_\varrho$ be its closed
quadratic form.  Proposition~\ref{prop:uniform-gap} gives
\begin{equation}\label{eq:response-coercivity}
 \mathfrak q_\varrho[h]\geq\mu_*\|h\|_2^2
\end{equation}
on the corresponding form domain, uniformly in $\varrho$.  Also,
$\phi_\varrho'=\phi_0'(\,\cdot+c_\varrho)$ for $\varrho<1$ and
$\phi_1'=0$, so $\sup_\varrho\|\phi_\varrho'\|_2<\infty$.
The representation theorem for coercive closed forms gives a unique
\begin{equation}\label{eq:response-resolvent}
 \chi_\varrho=-L_\varrho^{-1}\phi_\varrho'.
\end{equation}
Taking $\chi_\varrho$ as a test function in its weak equation and using
\eqref{eq:response-coercivity} yields uniform $L^2$ bounds.  Since
$f'$ is bounded on $[0,1]$ and the Robin boundary term is nonnegative, the same
identity then gives a uniform bound for $\|\chi_\varrho'\|_2$.  The equation
gives the corresponding local $H^2$ bounds.  Positivity of the resolvent, or
equivalently the maximum principle applied to $-\chi_\varrho$, proves
$\chi_\varrho\leq0$.

If $\varrho_n\to\varrho$, Proposition~\ref{prop:translation} gives convergence
of the potentials in $L^\infty$ and convergence of
$\phi_{\varrho_n}'$ in $L^2$.  The latter follows directly from the translation
formula: at the Neumann endpoint it is the $L^2$ tail of $\phi_0'$, while at
every other parameter the translation lengths converge.  Lemma
\ref{lem:halfline-resolvent-continuity}, applied with
$V_n=-f'(\phi_{\varrho_n})$ and $g_n=-\phi_{\varrho_n}'$, shows that the
solutions in \eqref{eq:response-resolvent} converge in
$H^1\cap C_b^1$.  At $\varrho=1$ the right-hand side is zero, so uniqueness
gives $\chi_1=0$.

Uniform decay follows from the linearization at $1$ and the energy identity
for $\phi_0$.  They give constants
$C,c_0>0$ such that
\[
 |1-\phi_0(y)|+|\phi_0'(y)|+|\phi_0''(y)|\leq Ce^{-c_0y}.
\]
Since $\phi_\varrho=\phi_0(\,\cdot+c_\varrho)$ for $\varrho<1$ and
$\phi_1\equiv1$, the same bound holds uniformly in $\varrho$ with
$\phi_0$ replaced by $\phi_\varrho$.  By \eqref{eq:uniform-tail}, there
are $Y>0$ and $a>0$ such that
$-f'(\phi_\varrho(y))\geq a$ for $y\geq Y$, uniformly in $\varrho$.
The uniform $H^1$ bound already obtained and the one-dimensional Sobolev
inequality give
$\sup_{\varrho\in[0,1]}|\chi_\varrho(Y)|\leq C$.  Each
$\chi_\varrho\in H^1(\R_+)\cap C_b^1(\R_+)$ tends to zero at infinity.
By \eqref{eq:response-sign-continuity}, the set
$\{\chi_\varrho:\varrho\in[0,1]\}$ is compact in $C_b(\R_+)$; since it is
contained in $C_0(\R_+)$, this compactness gives
$\sup_\varrho|\chi_\varrho(y)|\to0$ as $y\to\infty$.
Choose $0<c<\min\{c_0,\sqrt a\}$.  Since
$0\leq\phi_\varrho'(y)\leq Ce^{-c_0y}$, a sufficiently large multiple of
$e^{-c(y-Y)}$ dominates both $-\chi_\varrho$ at $y=Y$ and the forcing in
\[
 \bigl(-\partial_y^2-f'(\phi_\varrho)\bigr)(-\chi_\varrho)
 =\phi_\varrho'.
\]
To justify the comparison at infinity, apply it first on $[Y,T]$ after adding
a constant $\varepsilon>0$ to the supersolution.  The uniform vanishing
at infinity lets $T$ be chosen so that
$-\chi_\varrho(T)\leq\varepsilon$ for every $\varrho$, while the added
constant contributes at least $a\varepsilon$ to the left-hand side.  Sending
first $T\to\infty$ and then $\varepsilon\downarrow0$ gives
$|\chi_\varrho(y)|\leq Ce^{-cy}$ uniformly in $\varrho$.  The
equation and the corresponding exponential bound for $\phi_\varrho'$ yield
$|\chi_\varrho''(y)|\leq Ce^{-cy}$.  Since
$\chi_\varrho'\in L^2(\R_+)$ and $\chi_\varrho''\in L^1(Y,\infty)$,
$\chi_\varrho'$ has a finite limit at infinity, and that limit must be zero.
Since this limit is zero,
$|\chi_\varrho'(y)|\leq\int_y^\infty|\chi_\varrho''(s)|\,ds\leq Ce^{-cy}$.
Decreasing $c$ if necessary, we obtain the uniform tail estimate
\begin{equation}\label{eq:profile-response-joint-tail}
 \sup_{\varrho\in[0,1]}
 \left(|1-\phi_\varrho(y)|+|\phi_\varrho'(y)|+|\phi_\varrho''(y)|
 +\sum_{j=0}^2|\chi_\varrho^{(j)}(y)|\right)
 \leq Ce^{-cy}.
\end{equation}

\end{proof}

We shall also use the following uniform $L^\infty$ resolvent estimate.

\begin{lemma}
\label{lem:uniform-bounded-resolvent}
Assume that $\Omega$ is bounded and uniformly $C^{2,\gamma}$, with nonempty
boundary.  There are $\delta>0$, $\kappa_2>0$, and $C>0$ such that the
following holds.  Let $\kappa\geq\kappa_2$, $\varrho\in[0,1]$, and
$q,g\in L^\infty(D_\kappa)$ satisfy
\begin{equation}\label{eq:resolvent-potential-close}
 \|q-f'(\Phi_{\kappa,\varrho})\|_\infty\leq\delta.
\end{equation}
If a bounded weak solution $w$ in the quadratic-form sense satisfies
\begin{equation}\label{eq:bounded-resolvent-equation}
 -\Delta w-qw=g\quad\text{in }D_\kappa,\qquad
 \cN_\varrho w=0\quad\text{on }\partial D_\kappa,
\end{equation}
then
\begin{equation}\label{eq:bounded-resolvent-estimate}
 \|w\|_\infty\leq C\|g\|_\infty.
\end{equation}
\end{lemma}

\begin{proof}
Fix $0<\delta<\min\{\mu_*/4,|f'(1)|/4\}$.  If no uniform estimate held for
this $\delta$, then for each integer $n\geq1$ we could choose
data $\kappa_n\geq n$, $\varrho_n\in[0,1]$, an admissible coefficient $q_n$,
and a bounded weak solution
$w_n^{\rm raw}$ with right-hand side $g_n^{\rm raw}$ such that
\[
 \|w_n^{\rm raw}\|_\infty>n\|g_n^{\rm raw}\|_\infty.
\]
Divide the equation by $\|w_n^{\rm raw}\|_\infty$ and pass to a subsequence
with $\varrho_n\to\varrho_*$.  Since $f'$ is bounded on $[0,1]$, the coefficients $q_n$ are uniformly bounded.
We then have
\begin{equation}\label{eq:resolvent-contradiction-sequence}
 \|w_n\|_\infty=1,\qquad \|g_n\|_\infty\leq n^{-1},\qquad
 \|q_n-f'(\Phi_{\kappa_n,\varrho_n})\|_\infty\leq\delta.
\end{equation}
For each $n$, standard boundary regularity gives a continuous representative
of $w_n$ on $\overline{D_{\kappa_n}}$.  Since $D_{\kappa_n}$ is bounded, choose
$x_n$ there with $|w_n(x_n)|=1$.

Suppose first that, along a subsequence,
$d_{\kappa_n}(x_n)\to\infty$.  After translating $x_n$ to
the origin, every fixed ball eventually lies inside the translated domain, and
$q_nw_n+g_n$ is uniformly bounded there.  Interior $W^{2,p}$ estimates with $p>d$, followed by a diagonal extraction,
give locally uniform convergence $w_n\to w$ and weak-star convergence $q_n\stackrel{*}{\rightharpoonup}
q_\infty$.  Since $|w_n(0)|=1$, one has $|w(0)|=1$.  The strong local
convergence of $w_n$ and weak-star convergence of $q_n$ allow passage to the
product in distributions, so
$-\Delta w-q_\infty w=0$ on $\R^d$.  On each fixed translated ball $B_R(0)$ one has
$d_{\kappa_n}(x_n+x)\ge d_{\kappa_n}(x_n)-R$, so the uniform profile tail makes
$f'(\Phi_{\kappa_n,\varrho_n})\to f'(1)$ there.  Together with
\eqref{eq:resolvent-contradiction-sequence}, this implies
$-q_\infty\geq |f'(1)|-\delta>0$.  The form bottom of
$-\Delta-q_\infty$ on $\R^d$ is positive, and
Lemma~\ref{lem:bounded-weak-maximum} applied to $w$ and $-w$ forces $w=0$,
contradicting $|w(0)|=1$.

Hence the maximizing points remain at bounded distance from the boundary.
Choose nearest boundary points $p_n$, translate $p_n$ to the origin, and rotate
the inward normal to $e_d$, as in Proposition~\ref{prop:coalescence}.  After
extraction, the transformed points $x_n$ converge to some
$x_*\in\overline{\Hh^d}$.

If $\varrho_*>0$, then $\beta_{\varrho_n}$ stays bounded, and standard
local Robin/Neumann $W^{2,p}$ estimates give, after extraction,
$w_n\to w_*$ locally uniformly on $\overline{\Hh^d}$.

Suppose $\varrho_*=0$.  If infinitely many indices have $\varrho_n=0$,
pass to that subsequence; the usual local Dirichlet estimates give compactness
up to the boundary.  Otherwise $\varrho_n>0$ eventually and
$\beta_n:=\beta_{\varrho_n}\to\infty$.  We claim that the traces of $w_n$
vanish locally.  If not, there are boundary points $y_n$ in
one fixed chart and an $\varepsilon>0$ with $|w_n(y_n)|\geq\varepsilon$.
Rescale at $y_n$ by $\beta_n^{-1}$ exactly as in
Lemma~\ref{lem:joint-compactness}.  The rescaled functions are bounded by one
and satisfy
\[
 -\Delta \widetilde w_n
 =\beta_n^{-2}(q_nw_n+g_n),\qquad
 \partial_{\widetilde\nu_n}\widetilde w_n+\widetilde w_n=0.
\]
The right-hand side tends uniformly to zero on fixed patches, so the oblique
compactness argument from Lemma~\ref{lem:joint-compactness} would produce a
nonzero bounded harmonic limit with $\partial_\nu w+w=0$ if the traces did not
vanish.  Lemma~\ref{lem:harmonic-robin} excludes this.  Lemma~\ref{lem:boundary-decay}
then combines with interior compactness to give local uniform convergence up to
the limiting Dirichlet boundary.
Thus, after extraction, there is a bounded limit $w_*$ on $\Hh^d$ with the
$\varrho_*$ boundary condition.  Since the convergence holds on the
closure near $x_*$,
\[
 |w_*(x_*)|=\lim_{n\to\infty}|w_n(x_n)|=1,
\]
so $w_*$ is nonzero.

Write
\[
 q_n=f'(\Phi_{\kappa_n,\varrho_n})+r_n,
 \qquad \|r_n\|_\infty\leq\delta.
\]
On every fixed half-ball the first term converges uniformly to
$f'(\phi_{\varrho_*})$.  After a diagonal weak-star extraction on half-balls, these local limits define
$r_*\in L^\infty(\Hh^d)$ with $\|r_*\|_\infty\leq\delta$.  Hence
$q_*=f'(\phi_{\varrho_*})+r_*$ and therefore
\[
 \|q_*-f'(\phi_{\varrho_*})\|_{L^\infty(\Hh^d)}\leq\delta.
\]
Local uniform convergence of $w_n$ and weak-star convergence of $q_n$ allow
passage to the equation in distributions:
$-\Delta w_*-q_*w_*=0$.
The potential perturbation estimate, Lemma~\ref{lem:product}, and
\eqref{eq:uniform-gap} give
\[
 \lambda(-\Delta-q_*,\Hh^d,\varrho_*)
 \geq\mu_{\varrho_*}-\delta>0.
\]
The maximum principle of Lemma~\ref{lem:bounded-weak-maximum}, applied to
$w_*$ and $-w_*$, then gives $w_*=0$.  The normalization at the limiting
maximizing point excludes this, and \eqref{eq:bounded-resolvent-estimate}
follows.

\end{proof}

\begin{proof}[Proof of Theorem~\ref{thm:curvature-correction}]
Extend $f$ to a $C^{1,\gamma}$ function on $\R$.  Only values in a fixed
neighborhood of $[0,1]$ enter the Taylor expansion, and the exact solution
remains in $(0,1]$, so the extension is immaterial.

Let $r_*>0$ be a tubular radius for $\partial\Omega$.  For
$x\in D_\kappa$ with $d_\kappa(x)<\kappa r_*$, write uniquely
\begin{equation}\label{eq:normal-coordinates}
 x=\kappa p-y\nu(p),\qquad p\in\partial\Omega,\quad
 y=d_\kappa(x),
\end{equation}
and set $\mathcal H_\kappa(x)=\mathcal H(p)$.  If
$k_1(p),\ldots,k_{d-1}(p)$ are the outward principal curvatures, the normal
coordinate formula gives
\begin{equation}\label{eq:distance-laplacian-curvature}
 \Delta d_\kappa(x)
 =-\frac1\kappa\sum_{j=1}^{d-1}
   \frac{k_j(p)}{1-yk_j(p)/\kappa}.
\end{equation}
The signed-distance regularity in this tubular neighborhood follows, for
example, from \cite{KrantzParksDistance}; the displayed Laplacian identity is
the standard principal-coordinate calculation.  Uniformly for
$0\leq y\leq M\log\kappa+1$,
\begin{equation}\label{eq:distance-curvature-error}
 \left|\Delta d_\kappa+\frac{\mathcal H_\kappa}{\kappa}\right|
 \leq C\frac{1+y}{\kappa^2},\qquad
 \nabla d_\kappa\cdot\nabla\mathcal H_\kappa=0,\qquad
 |\Delta\mathcal H_\kappa|\leq C\kappa^{-2}.
\end{equation}
The last estimate uses the $C^{4,\gamma}$ regularity of the boundary.  In
dimension one all curvature sums and the function $\mathcal H_\kappa$ are zero.

Choose $\eta\in C^\infty(\R)$ with $0\leq\eta\leq1$,
$\eta(t)=1$ for $t\leq0$, and $\eta(t)=0$ for $t\geq1$.  We cut off the boundary expansion on a logarithmic strip, using the
exponential decay of the profiles.  From \eqref{eq:profile-response-joint-tail}, choose
$M>3/c$ and set
$L_\kappa=M\log\kappa$.  For all sufficiently large $\kappa$ one then has,
uniformly in $\varrho$,
\begin{equation}\label{eq:cutoff-tail-kappa3}
 \sup_{y\geq L_\kappa}
 \left(|1-\phi_\varrho(y)|+|\phi_\varrho'(y)|+|\phi_\varrho''(y)|
 +\sum_{j=0}^2|\chi_\varrho^{(j)}(y)|\right)
 \leq C\kappa^{-3}.
\end{equation}
For large $\kappa$, also $L_\kappa+1<\kappa r_*/2$.

Now set
\begin{equation}\label{eq:corrected-approximate-state}
 A_{\kappa,\varrho}(x):=
 1+\eta(y-L_\kappa)\bigl(\phi_\varrho(y)-1\bigr)
 +\frac{\mathcal H_\kappa(x)}{\kappa}
       \eta(y-L_\kappa)\chi_\varrho(y)
\end{equation}
in the tubular neighborhood, and by $A_{\kappa,\varrho}=1$ elsewhere.
The cutoff vanishes before the edge of the tubular neighborhood, so this is a
global $C^2$ function.  The normal extension of
$\mathcal H$ is constant on normal rays, and both $\phi_\varrho$ and
$\chi_\varrho$ satisfy the same homogeneous boundary condition, so
\begin{equation}\label{eq:approximate-boundary-condition}
 \cN_\varrho A_{\kappa,\varrho}=0
 \quad\text{on }\partial D_\kappa
\end{equation}
for every $\varrho\in[0,1]$.

Its defect is
\[
 \mathcal R_{\kappa,\varrho}:=
 -\Delta A_{\kappa,\varrho}-f(A_{\kappa,\varrho}).
\]
In the core region $0\leq y\leq L_\kappa$, put
$h=\kappa^{-1}\mathcal H_\kappa\chi_\varrho$.  The
$C^{1,\gamma}$ Taylor estimate gives
\begin{equation}\label{eq:holder-taylor-remainder}
 |f(\phi_\varrho+h)-f(\phi_\varrho)-f'(\phi_\varrho)h|
 \leq C|h|^{1+\gamma}.
\end{equation}
Using $-\phi_\varrho''=f(\phi_\varrho)$,
$L_\varrho\chi_\varrho=-\phi_\varrho'$, and
$\nabla d_\kappa\cdot\nabla\mathcal H_\kappa=0$, direct differentiation gives
\begin{align*}
 \mathcal R_{\kappa,\varrho}
 ={}&-\phi_\varrho'\Delta d_\kappa
 +\frac{\mathcal H_\kappa}{\kappa}
   \bigl(-\chi_\varrho''-f'(\phi_\varrho)\chi_\varrho\bigr)\\
 &-\frac{\mathcal H_\kappa}{\kappa}
       \chi_\varrho'\Delta d_\kappa
 -\frac{\chi_\varrho}{\kappa}\Delta\mathcal H_\kappa
 +O\!\left(\kappa^{-1-\gamma}e^{-c(1+\gamma)y}\right).
\end{align*}
With the sign and forcing in the definition of $\chi_\varrho$, the
first-order curvature terms cancel.  Since
$L_\varrho\chi_\varrho=-\phi_\varrho'$,
\[
 -\phi_\varrho'\Delta d_\kappa
 +\frac{\mathcal H_\kappa}{\kappa}
   \bigl(-\chi_\varrho''-f'(\phi_\varrho)\chi_\varrho\bigr)
 =-\phi_\varrho'
   \left(\Delta d_\kappa+\frac{\mathcal H_\kappa}{\kappa}\right).
\]
On the core strip, \eqref{eq:distance-curvature-error} also gives
$|\Delta d_\kappa|\leq C\kappa^{-1}$ for large $\kappa$.  The
preceding identity contributes at most
$C\kappa^{-2}(1+y)e^{-cy}$; the term containing
$\chi_\varrho'\Delta d_\kappa$ is $O(\kappa^{-2}e^{-cy})$, and the
$\Delta\mathcal H_\kappa$ term is $O(\kappa^{-3}e^{-cy})$.  Also,
$|h|\leq C\kappa^{-1}e^{-cy}$, so the Taylor remainder is
$O(\kappa^{-1-\gamma}e^{-c(1+\gamma)y})$.  Hence
\begin{equation}\label{eq:core-residual-bound}
 |\mathcal R_{\kappa,\varrho}|
 \leq C\bigl(\kappa^{-1-\gamma}
              +\kappa^{-2}(1+y)e^{-cy}\bigr)
 \qquad(0\leq y\leq L_\kappa).
\end{equation}
In the cutoff region $L_\kappa\leq y\leq L_\kappa+1$, every derivative
of the cutoff is bounded while all profile and response deviations appearing
after differentiating are $O(\kappa^{-3})$ by
\eqref{eq:cutoff-tail-kappa3}; since $f(1)=0$, the nonlinear term has the same
order there.  The cutoff-region defect is $O(\kappa^{-3})$, and outside the
tube it vanishes.  Moreover, on the core strip
$A_{\kappa,\varrho}-\Phi_{\kappa,\varrho}
=\kappa^{-1}\mathcal H_\kappa\chi_\varrho$, whereas beyond the core the
uniform tail estimate gives an even smaller error.  Hence, uniformly in
$\varrho$,
\begin{equation}\label{eq:global-residual-bound}
 \|\mathcal R_{\kappa,\varrho}\|_\infty
 \leq C\bigl(\kappa^{-1-\gamma}
              +\kappa^{-2}\bigr),
 \qquad
 \|A_{\kappa,\varrho}-\Phi_{\kappa,\varrho}\|_\infty
 \leq C\kappa^{-1}.
\end{equation}

Let $u_{\kappa,\varrho}$ be the unique state and put
$w=u_{\kappa,\varrho}-A_{\kappa,\varrho}$.  Since
$0\leq\Phi_{\kappa,\varrho}\leq1$ and
\eqref{eq:global-residual-bound} gives
$\|A_{\kappa,\varrho}-\Phi_{\kappa,\varrho}\|_\infty=O(\kappa^{-1})$, the
segment joining $A_{\kappa,\varrho}(x)$ to $u_{\kappa,\varrho}(x)\in(0,1]$
stays in one fixed neighborhood of $[0,1]$ for all large $\kappa$.  Hence
only values of the extension of $f$ in this neighborhood are involved.  Define
\[
 q_{\kappa,\varrho}(x):=
 \int_0^1 f'\bigl(A_{\kappa,\varrho}(x)
       +t w(x)\bigr)\,dt.
\]
Proposition~\ref{prop:coalescence}, \eqref{eq:global-residual-bound}, and
uniform continuity of $f'$ give
\begin{equation}\label{eq:corrector-linearization-close}
 \sup_{\varrho\in[0,1]}
 \|q_{\kappa,\varrho}-f'(\Phi_{\kappa,\varrho})\|_\infty
 \longrightarrow0.
\end{equation}
Subtracting the equations for $u_{\kappa,\varrho}$ and
$A_{\kappa,\varrho}$ gives
\[
 -\Delta w-q_{\kappa,\varrho}w=-\mathcal R_{\kappa,\varrho},
 \qquad \cN_\varrho w=0.
\]
Lemma~\ref{lem:uniform-bounded-resolvent},
\eqref{eq:corrector-linearization-close}, and
\eqref{eq:global-residual-bound} imply
\begin{equation}\label{eq:global-corrector-error}
 \sup_{\varrho\in[0,1]}
 \|u_{\kappa,\varrho}-A_{\kappa,\varrho}\|_\infty
 \leq C\bigl(\kappa^{-1-\gamma}
              +\kappa^{-2}\bigr).
\end{equation}
On every fixed boundary strip, $R<L_\kappa$ for large $\kappa$, so
\eqref{eq:corrected-approximate-state} is uncut there.  Then
\eqref{eq:global-corrector-error} gives
\eqref{eq:intro-curvature-expansion}, and multiplication by $\kappa$ gives
\eqref{eq:intro-scaled-curvature-limit}.
\end{proof}

\begin{corollary}
\label{cor:quantitative-solution-stability}
Under the hypotheses of Theorem~\ref{thm:curvature-correction}, there are
$C>0$ and $\kappa_{\rm nl}>0$ such that
\begin{equation}\label{eq:quantitative-solution-stability}
 \sup_{\varrho\in[0,1]}
 \left|
 \lambda\bigl(-\Delta-f'(u_{\kappa,\varrho}),D_\kappa,\varrho\bigr)
 -\mu_\varrho
 \right|
 \leq C\bigl(\kappa^{-1/2}+\kappa^{-\gamma}\bigr)
 \qquad(\kappa\geq\kappa_{\rm nl}).
\end{equation}
\end{corollary}

\begin{proof}
From \eqref{eq:global-residual-bound} and
\eqref{eq:global-corrector-error} we already have the global estimate
\[
 \sup_{\varrho\in[0,1]}
 \|u_{\kappa,\varrho}-\Phi_{\kappa,\varrho}\|_\infty
 \leq C\kappa^{-1}.
\]
Since $f'$ is $C^{0,\gamma}$, the perturbation estimate
\eqref{eq:potential-lipschitz} shows that passing from
$f'(\Phi_{\kappa,\varrho})$ to $f'(u_{\kappa,\varrho})$ changes the spectral
bottom by at most $C\kappa^{-\gamma}$.  Theorem~\ref{thm:spectral-convergence} supplies the remaining
$C\kappa^{-1/2}$ error, and the stated bound follows.
\end{proof}

\end{document}